\documentclass[11pt]{amsart}
\usepackage{amsmath}
\usepackage{amssymb}
\usepackage[mathscr]{euscript}
\newtheorem{prop}{Proposition}[section]

\newtheorem{defn}[prop]{Definition}

\newtheorem{lem}[prop]{Lemma}
\newtheorem{cor}[prop]{Corollary}

\newtheorem{thm}[prop]{Theorem}

\newtheorem{examp}[prop]{Example}

\newtheorem{que}[prop]{Question}
\numberwithin{equation}{section}

\def\C{\mathbb{C}}
\def\E{\mathbb{E}}
\def\Z{\mathbb{Z}}
\def\Q{\mathbb{Q}}
\def\R{\mathbb{R}}
\def\N{\mathbb{N}}
\def\P{\mathbb{P}}
\def\X{\mathscr{X}}
\def\Y{\mathscr{Y}}
\def\V{\mathscr{V}}
\def\W{\mathscr{W}}
\def\M{\mathcal{M}}
\def\LP{\left(}
\def\RP{\right)}

\def\Spec{{\rm Spec}\,}

\def\F{\mathbb{F}}

\def\Aut{{\rm Aut}}

\def\rk{{\rm rk}}

\def\con#1#2#3{\hbox{$#1\equiv#2$ \rm(mod $#3$)}}

\def\vt{\vartheta}
\begin{document}

\title[Multiplicative Series and Modular Forms]{Multiplicative Series, Modular Forms, and Mandelbrot Polynomials}

\author{Michael Larsen}
\address{Department of Mathematics\\
    Indiana University \\
    Bloomington, IN 47405\\
    U.S.A.}
\email{mjlarsen@indiana.edu}

\thanks{The author was partially supported by the Sloan Foundation and the NSF}
\begin{abstract}
We say a  power series $\sum_{n=0}^\infty a_n q^n$ is \emph{multiplicative} if the sequence $1,a_2/a_1,\ldots,a_n/a_1,\ldots$ is so.
In this paper, we consider multiplicative power series $f$ such that $f^2$ is also multiplicative.
We find a number of examples for which $f$ is a rational function or a theta series and prove that the complete set of solutions is
the locus of a (probably reducible) affine variety over $\C$.   
The precise determination of this variety turns out to be a finite computational problem, but it seems to be beyond the reach of current computer algebra systems.
The proof of the theorem depends on a bound on the logarithmic capacity of the
Mandelbrot set.  
\end{abstract}
\maketitle

\section{Introduction}

Let $r_k(n)$ denote the number of representations of $n$ as a sum of $k$
squares.  It is classical that $r_1(n)/2$, $r_2(n)/4$, $r_4(n)/8$, and 
$r_8(n)/16$ are multiplicative functions of $n$; the first trivially, 
the second thanks to Fermat, and the third and fourth thanks to Jacobi
\cite[\S\S42,44]{Jacobi}.
From the standpoint of generating functions, this can be interpreted as 
the statement that the theta series $\vt_{\Z}(q)$ and its square, fourth power, and eighth power, 
all have multiplicative
coefficients (after suitable 
normalization).  As a starting point, we prove the converse:

\begin{thm}
\label{eight}
If $f(q)\in\C[[q]]$, $f(q)^2$, $f(q)^4$, and $f(q)^8$ are all multiplicative,
then $f(q) = c\vt_{\Z}(\pm q)$.
\end{thm}

This is an immediate consequence of the following more difficult result:

\begin{thm}
\label{four}
If $f(q)$, $f(q)^2$, and $f(q)^4$ all have multiplicative coefficients,
then $f(q)$ is a constant multiple of $\vt_{\Z}(\pm q)$, $\vt_{\Z[i]}(\pm q)$, or $\vt_{\Z[\zeta_3]}(\pm q)$.
\end{thm}

A much more difficult problem is to characterize all power series $f(q)$ such that $f(q)$ and $f(q)^2$ are multiplicative, without assuming $f(q)^4$ is multiplicative as well.
We denote by $\X$ the set of normalized multiplicative power series $f$ such that $f^2$ is also multiplicative.  (See Definition~\ref{normal} for precise definitions.)
Since a power series with multiplicative coefficients is
determined by its prime power coefficients, and since prime powers
form a density-zero subset of the integers, when $n$ is large, the first $n$ coefficients of any $f(q)\in \X$ must satisfy a highly overdetermined system of polynomial equations.
From this point of view, the fact that $\X\neq \emptyset$ is surprising.
On the other hand, it
is clear that any Hecke eigenform whose square is again a Hecke eigenform
belongs to $\X$.  The relationship 
between the action of Hecke operators on the space of modular forms of 
a fixed weight and the ring structure on the graded vector space of all 
modular forms is rather mysterious, but, in general,
the square of an eigenform is unlikely to be an eigenform unless there is no 
alternative.  Indeed, a number of papers have examined when the product of two eigenforms is again an eigenform
(see, e.g., \cite{BJTX, Duke, Emmons, Ghate, Johnson}), and the moral of these papers seems to be that this phenomenon is a transient one, associated to low levels and weights.
A natural place to look for elements of $\X$ is therefore 
among (noncuspidal) forms of low level and weight.  One might reasonably guess that $\X$ consists entirely of modular forms, but this turns out to be wrong; 
certain rational functions
analytic on the open unit disk also belong to $\X$.

Between modular forms and rational functions,
the locus $\X$ contains at least nine one-parameter families of 
solutions and twelve isolated points.    There is numerical evidence, based on 
the search for (mod $p$) solutions for small primes $p$, 
that these solutions
constitute all of $\X$, but the results of this paper fall far short of this.
Truncating power series at the $q^n$ coefficient, as $n$ varies, one obtains a sequence
of complex algebraic varieties $\X_n$ of which $\X$ is the inverse limit.  
The sequence $\X_n$ does not stabilize.  
The main theorem of this paper asserts that, nevertheless, $\X$ itself has the structure of a complex 
affine variety.  More precisely, there exists $n$ (in fact, $n=16$ will do) such that the natural map $\X\to \X_n$ is injective,
and its image is a Zariski-closed subset of $\X_n$.  In particular, all solutions are determined by their degree $16$ truncations.

Remarkably, our proof of this theorem depends, ultimately, on the fact that the logarithmic
capacity of the Mandelbrot set $\M$ is less than $2$.  (In fact, it is known to be $1$ \cite[\S 6.2]{Steinmetz}.)
Computer algebra computations reduce the problem to  the ``sparse'' case, where the coefficients
$a_2=a_3=\cdots=a_{n-1}=0$ and $a_n\neq 0$, for some $n\ge 16$.  In this case, one shows first
that $n$ is a Mersenne prime and then that  the first $\frac {n-1}2$ terms of
the sequence $a_n, a_{2n-1}, a_{3n-2}, a_{4n-3},\ldots$ satisfy a certain non-linear recurrence.
In fact, there is a universal sequence $M_1,M_2,\ldots\in \Q[y]$ of ``Mandelbrot polynomials'' such that
$a_{i(n-1)+1} = M_i(a_1)$.  The multiplicativity of the sequence of coefficients implies that
if $i(n-1)+1$ is not a prime power,  then $a_{i(n-1)+1} = 0$.
The recurrence formula for the $M_i$ implies  that if $r_i$ is a root of $M_i(y)$ for each $i$ and $r$ is a limit point of
the sequence $r_i$, then 
$-2r$ belongs to the Mandelbrot set.  
Although the roots
of the individual $M_i(y)$ need not be algebraic integers, we have enough
$p$-adic control to guarantee integrality for a simultaneous
root of many $M_i$, like $a_n$.  Since $\X$ is $\Aut(\C)$-stable
and 
a set of capacity less than $1$ contains only finitely many complete Galois orbits
of algebraic integers, there are only finitely many possibilities for $a_n$,
and in the end we show $a_n=0$.  We actually work not with $\M$ itself but with an open disk containing $\M$ and of radius $<2$,
thus obviating the need to understand the fine structure of $\M$.

The paper is organized as follows.
In \S2, we give some preliminaries about inverse systems of varieties.
In \S3, we present the known elements of $\X$.
In \S4, we
assemble elementary results about the set of prime powers which are needed in the next two sections.
In \S5, we present results of Maple-assisted computations which reduce the problem to the sparse case.
In principle, the results of this section imply that the non-sparse solutions 
can be determined by a finite 
computation, but this seems well beyond
the reach of currently available computer algebra systems. 
Sparse solutions are ruled out in \S6 by the method discussed above.
The last section presents
variants and related questions, including proofs of Theorems \ref{eight} and \ref{four}.
An appendix presents the results of an exhaustive search for non-sparse solutions defined over the finite field $\F_p$ for small primes $p>2$.
I am grateful to Anne Larsen for carrying out this search and identifying almost all of her solutions as modular forms, including a number of ``exceptional solutions'',
that is, examples in which the form involved is not the (mod $p$) reduction of any known characteristic zero solution.  The solutions (x) and (x'), which did not appear in an
earlier draft of this paper, originally appeared as exceptional solutions in her (mod $p$) tables for $p=3$, $p=11$, $p=17$, and $p=19$.

I would like to thank the Hebrew University in Jerusalem for its hospitality
while much of this work was carried out.  I am grateful to Ze\'ev Rudnick for pointing out some relevant literature.

\section{Systems of Affine Varieties}

Throughout this paper, an \emph{affine variety} means a scheme $\V = \Spec A$, where $A$ is a finitely
generated algebra over $\C$, and a \emph{morphism of varieties} means a morphism over $\Spec \C$.  When no confusion seems likely to result, we identify $\V$ with its set $\V(\C)$ of closed points.

Let $(\Y_n, \phi_{m,n}\colon \Y_m\to \Y_n)$ denote an inverse system of affine varieties 
indexed by integers $n\ge 2$.  Let $(\Y= \varprojlim \Y_n(\C), \psi_n\colon \Y\to \Y_n(\C))$ denote the set-theoretic inverse limit.

\begin{defn}
We say the inverse limit $\Y$ is of \emph{affine type} if there exist $n$ and a closed subvariety $\V_n$ of $\Y_n$ such that $\psi_n$ is injective and
$\psi_n(\Y) = \V_n(\C)$.
\end{defn}

\begin{examp}
If $\Y_n = \Spec \C[x]$ for all $n$, and every map $\phi_{m,n}$ comes from the $\C$-algebra homomorphism $\C[x]\to \C[x]$ mapping $x$ to $0$,
then $\Y$ is of affine type (and consists of a single point).
\end{examp}

\begin{examp}
If $\Y_n = \Spec \C[x,\frac 1 {x-1}, \frac 1 {x-2}, \ldots, \frac 1 {x-n}]$ with the obvious inclusion morphisms, then $\Y = \C \setminus \Z^{>0}$ is not of affine type. 
\end{examp}

\begin{examp}
If $\Y_n = \Spec \C[x]/(x^{2^n}-1)$ with the obvious morphisms, then $\Y=\Z_2$ is not of affine type. 
\end{examp}

\begin{prop}
\label{AffineType}
Let $(\Y_n, \phi_{m,n}\colon \Y_m\to \Y_n)$ be an inverse system, $n\ge 2$ an integer, and $\V$ a closed subvariety of $\Y_n$.
Assume the following conditions hold:
\begin{enumerate}
\item For all $y\in \Y_n$, $|\psi_n^{-1}(y)| \le 1$, 
\item $\V$ is contained in $\psi_n(\Y)$,
\item For all $y\in \Y_n\setminus \V$ and all sufficiently large $m$, $|\phi_{m,n}^{-1}(y)|\le 1$.
\item For all $y\in \Y_n\setminus \V$, there exists a neighborhood $U_y$ of $y$ in the complex topology
such that there exist arbitrarily large integers $m$ for which $\phi_{m,n}^{-1}(U_y)$ is precompact in the complex topology.
\end{enumerate}
Then $\Y$ is of affine type.
\end{prop}

\begin{proof}
For $m > n$, let $\W_m$ denote the Zariski closure of $\phi_{m,n}(\Y_m)$.  Thus,
$$\Y_n\supseteq \W_{n+1}\supseteq \W_{n+2}\supseteq \W_{n+3}\supseteq\cdots,$$
and by the Hilbert basis theorem, this chain must eventually stabilize to some closed subvariety $\W_k\subseteq \Y_n$.
We define $\V_n := \W_k$.
Thus, $\V\subseteq \psi_n(\Y)\subseteq \V_n.$  We need only prove that for all $y\in \V_n\setminus \V$, the inverse image
$\psi_n^{-1}(y)\subset \Y$ is non-empty.  As $\phi_{m,n}^{-1}(y)$ is finite for all $y\in \Y_n\setminus \V$ and for all $m$ sufficiently large,
and since the inverse limit of an inverse system of non-empty finite sets is non-empty, it suffices to prove that
$\phi_{m,n}^{-1}(y)$ is non-empty for all $y\in \V_n\setminus \V$ and all $m$ sufficiently large.

As $\phi_{m,n}(\Y_m)$ contains a Zariski dense open subset in $\Y_n$, it contains an open set $U_m$ in the complex topology.
Intersecting with the open set $U_y$ and choosing $m$ larger if necessary, 
we may assume that $\phi_{m,n}^{-1}(U_m)$ is precompact in the complex topology.
Now, $y$ is the limit in the complex topology of a sequence of points $y_i\in \phi_{m,n}(\Y_m)$. Choosing $\tilde y_i\in \Y_m$ such that
$\phi_{m,n}(\tilde y_i) = y_i$, the $\tilde y_i$ belong to a precompact set, so some subsequence converges to $\tilde y\in \Y_m$, and
it follows that $\phi_{m,n}(\tilde y) = y$.
\end{proof}

\section{Solutions}

\begin{defn}
\label{normal}
A power series $f(q)=\frac 1{2a_0}+\sum_{n=1}^\infty
a_nq^n$ is
\emph{normalized multiplicative} if $a_1=1$ and $a_{mn}=a_m a_n$
whenever $m$ and $n$ are relatively prime.  We say that $f$ is
\emph{multiplicative} if some multiple $\lambda f$
is normalized multiplicative.  The set $\X$ consists of all normalized
multiplicative power series $f(q)$ such that $f(q)^2$ is again multiplicative.
\end{defn}

Equivalently $f$ is multiplicative if and only if the corresponding
Dirichlet series has an Euler product
$$\sum_{n=1}^\infty a_n n^{-s} = c\prod_p (1+a_p p^{-s}+a_{p^2} p^{-2s}+\cdots),$$
and normalized if $c=1$.

If $f(q)$ is the $q$-expansion of a modular form of prime-to-$p$ level
and $T_p f = \lambda f$ for some $\lambda$, then the Dirichlet series
for $f$ is the product of a $p$-factor 
$(1-\lambda p^{-s}+\epsilon(p) p^{2k-1-s})^{-1}$ and a prime-to-$p$
Dirichlet series.  In particular, any Hecke eigenform of prime-power level
is multiplicative. 
For general level $N$, if $f$ is an eigenform also for the Atkin-Lehner
operators, then it is again multiplicative.

It is convenient to express the modular solutions as theta-functions.  Thus,
if $\Lambda$ is a lattice all of whose elements have integral square-length,
we write
$$\vt_\Lambda(q)={1\over |\{\lambda\in\Lambda\mid \Vert\lambda\Vert=1\}|}
\sum_{\lambda\in\Lambda} q^{\Vert\lambda\Vert^2}$$

\begin{prop}
\label{KnownSolutions}
Let $\Phi$ denote the root lattice of the Lie algebra
$E_8$ normalized so that roots have length 1, $H$ the Hurwitz order in the rational quaternion algebra, and $D^*$ 
the set $\C\setminus\{-1\}$.  Then the following modular forms
lie in $\X$:
\begin{itemize}
\item[(i)] $\vt_\Phi(q)$,
\item[(ii)] $\vt^2_{\Z[\zeta_3]}(q)$,
\item[(iii)] $\vt_{H}(q)+t\vt_{H}(q^2),\ t\in D^*$,
\item[(iv)] $\vt_{\Z[\tau]}(q),\ \tau={1+\sqrt{-7}\over2}$,
\item[(v)] $\vt_{\Z[i]}(q)+t\vt_{\Z[i]}(q^2),\ t\in D^*$,
\item[(vi)] $\vt_{\Z[\zeta_3]}(q)+t\vt_{\Z[\zeta_3]}(q^4),\ t\in D^*$,
\item[(vii)] $\vt_{\Z}(q)+t\vt_{\Z}(q^4),\ t\in D^*$,
\item[(viii)] $\vt_{\Z[\sqrt{-2}]}(q)+t\vt_{\Z[\sqrt{-2}]}(q^2),
\ t\in D^*$,
\item[(ix)] $\vt_{\Z[i]}(q)-\sqrt{-3}\vt_{\Z[i]}(q^2)+\sqrt{-3}\vt_{\Z[i]}(q^3)
+3\vt_{\Z[i]}(q^6)+t(\vt_{\Z[i]}(q^2)+\sqrt{-3}\vt_{\Z[i]}(q^6)),
\ t\in D^*$,
\item[(x)] $\vt_{\Z[\zeta_3]}(q)+\sqrt{-2}\vt_{\Z[\zeta_3]}(q^2)$.
\end{itemize}
\end{prop}

\begin{proof}
We begin with a few general remarks.  If $f(q)$ is multiplicative and 
$n$ is a prime power, then $f(q)+tf(q^n)$ is multiplicative for all $t\in D^*$.
By \cite[p.~792]{Hecke}, if $R$ is the ring of integers in an imaginary quadratic field,
then $\vt_R(q)$ is a modular form of weight 1, level ${\rm Disc}(R)$, and
nebentypus of order 2.  If, in addition, $R$ is a PID, then
the corresponding theta-series is multiplicative.  This remark applies to 
$\Z[\zeta_3]$, $\Z[i]$, $\Z[{1+\sqrt{-7}\over 2}]$, and
$\Z[\sqrt{-2}]$.  The modular curves $X_0(N)$ for $N\in\{1,3,4,7,8,12,16\}$
are all of genus 0, so every $\Gamma_0(N)$ modular form of weight 2 
with $N$ in this set is a linear combination of Eisenstein series.

We now consider the individual cases.  By \cite[VII, \S6.6]{Serre}, $\vt_\Phi$ is the 
Eisenstein series $\frac 1{240}E_4$, and since there is only one normalized form
of level 1 and weight 8, $\vt_\Phi^2$ is an eigenform.  The
space of forms of weight 4 and level 3 (resp. 4) has dimension 2 (resp. 3)
and therefore consists entirely of linear combinations of Eisenstein series
(since the number of divisors of the level equals the dimension of the space).
This finishes (ii).  For (iii), we observe that $\vt_H(q)$ is of
level 2.  We can see this from the formula $\vt_H(q)=-\frac 1{24} E_2(q)+\frac 1{12}E_2(q^2)$
expressing the theta series of the Hurwitz order in terms of the 
not-quite-modular Eisenstein
series $E_2$ (cf. \cite[II \S5]{Mazur}).  Thus, the forms in question are all
of level $4$.
Case (vi) requires
extra care since unlike the cases (iv), (v), and (viii),
the level is no longer a prime power; we can write the Dirichlet series
for $f(q)^2$ as a product of $p$-factors for all $p\not\in\{2,3\}$ 
together with a factor
involving all terms of the form $2^m 3^n$.  As $b_{3n}=b_n$ for all $n$,
this final term is actually the product of $(1-3^{-s})^{-1}$ and a power
series in $2^{-s}$.
For (vii), the form $f$ is of weight $1/2$ and is multiplicative by
inspection.  By the two-squares theorem, its square is of the form 
$$\frac 14 \sum_{n=0}^\infty p(n)r_2(n)q^n,\quad
p(n)=\begin{cases}
1+t &\text{if } $\con n04$,\\
1 &\text{if } $\con n14$,\\
(1+t)^{-1} &\text{if } $\con n24$,\\
0 &\text{if } $\con n34$.
\end{cases}
$$
For (ix),
$$\vt_{\Z[i](q)}+u\vt_{\Z[i](q^2)}+v\vt_{\Z[i](q^3)}+uv\vt_{\Z[i](q^6)}$$
is multiplicative for all $u,v\in\C$.  As $X_0(24)$ has genus $1$, the
condition that $f(q)^2$ be a linear combination of Eisenstein series imposes
a single equation, which happens to be $v^2+3=0$.  We check that when
$v$ is a square root of $-3$, $b_{3n}=b_n$ for all $n$.
For (x), we verify
\begin{align*}
(1-\sqrt{-2})&(\vt_{\Z[\zeta_3]}(q)+\sqrt{-2}\vt_{\Z[\zeta_3]}(q^2))^2\\
&=E_2(q) - 2E_2(q^2) + (1+2\sqrt{-2})E_2(q^3) - (2+4\sqrt{-2})E_2(q^6),
\end{align*}
which is again multiplicative.
\end{proof}

We remark that (i)--(viii) above each have at least one
representative which is the theta-series of an order in a (possibly
non-commutative, possibly even non-associative) algebra.  This is obvious
except for (i), which corresponds to the ring of octavian integers in the Cayley numbers
(\cite[\S9.3]{CS}) and (ii), which corresponds to a maximal order in the rational 
quaternion algebra ramified only at $3$ and $\infty$.  We remark also
that (i), (ii), (iii), (v), (vi), and (vii) each contain at least one
representative which is the theta-series of a root lattice.\medskip

\begin{lem}
\label{SignChange}If $f(q)\in\X$, so is $-f(-q)$.
\end{lem}

\begin{proof}
We have $(-(-1)^m)(-(-1)^n)=(-(-1)^{mn})$ whenever $m$ and $n$
are not both even.  When they are both even, of course, they are
not relatively prime.
\end{proof}

\begin{cor}The following modular forms belong to $\X$:
\begin{itemize}
\item[(i')]$-\vartheta_\Phi(-q),$
\item[(ii')] $-\vartheta_{\Z[\zeta_3]}^2(-q)$,
\item[(iv')]$-\vartheta_{\Z[\tau]}(-q)$, 
\item[(ix')]$\vt_{\Z[i]}(q)-\sqrt{-3}\vt_{\Z[i]}(q^2)-\sqrt{-3}\vt_{\Z[i]}(-q^3)
+3\vt_{\Z[i]}(q^6)+t(\vt_{\Z[i]}(q^2)+\sqrt{-3}\vt_{\Z[i]}(q^6)),
\ t\in D^*$,
\item[(x')] $-\vt_{\Z[\zeta_3]}(-q)-\sqrt{-2}\vt_{\Z[\zeta_3]}(q^2)$.

\end{itemize}

\end{cor}

Next, we present some rational solutions.  Clearly, 

\begin{itemize}
\item[(xi)]$\frac 1{2a_0}+q,\ a_0\neq 0$
\end{itemize}
belongs to $\X$. 
It is easy to see that these are the only
polynomial solutions.  In addition, one readily checks the following
proposition: 

\begin{prop}The following rational functions all belong to $\X$:
\begin{itemize}
\item[(xii)]$t{1+q^2\over 1-q^2}+{q\over 1-q^2}=t+q+2tq^2+q^3+2tq^4+q^5
      +2tq^6+\cdots, t \neq 0$;
\item[(xiii)]${q^2+7q+1\over 6q^2+6q+6}={1\over 6}+q-q^2+q^4-q^5+q^7-q^8+\cdots$;
\item[(xiii')]${-q^2+7q-1\over 6q^2-6q+6}=-{1\over 6}+q+q^2-q^4-q^5+q^7+q^8-\cdots$;
\item[(xiv)]${q^2+10q+1\over 12(q-1)^2}={1\over 12}+q+2q^2+3q^3+4q^4+\cdots$;
\item[(xiv')]${-q^2+10q-1\over 12(q+1)^2}=-{1\over 12}+q-2q^2+3q^3-4q^4+\cdots$.  
\end{itemize}
\end{prop}
\medskip

The form of the above solutions suggests the following elementary proposition
whose proof we leave to the reader:

\begin{prop}
\label{RatForm}
If $f(q)$ is a multiplicative power
series which is a rational function but not a polynomial, then
there exists a constant $a_0$, a non-negative integer $d$, a positive
integer $N$,
and an $N$-periodic sequence of constants
$a_1,a_2,a_3,\ldots$ such that
$$f(q)=\frac 1{2a_0}+\sum_{n=1}^\infty a_n n^d q^n.$$
\end{prop}

The appendix presents the results of a comprehensive search for normalized multiplicative series $f(q)\in \F_p[[q]]$, $3\le p\le 31$, such that
$f(q)^2$ is again multiplicative.  The majority arise from (mod $p$) reduction of solutions (i)--(xiv') above (with $t\in\Q$ in the case of parametric solutions).
The exceptions appear to be (mod $p$) reductions of $q$-expansions of modular forms with coefficients in $\Q$
and in most cases can be written in the form $f(q) = \sum_{d\mid N} c_d g(q^d)$, where $g$ is either a theta series or an Eisenstein series.
However, somewhat unexpectedly, cusp forms also make an appearance.  The following proposition gives an illustrative example:

\begin{prop}
Let $\bar \Delta$ denote the (mod $13$) reduction of the normalized cusp form of level $1$ and weight $12$.  Then $\bar f(q) = 2+\bar \Delta$ is a normalized multiplicative series in $\F_{13}[[t]]$
whose square is again multiplicative.
\end{prop}

\begin{proof}
It is well known (see, for instance, \cite[VII,\ Corollary 2]{Serre}) that the ring of complex modular forms of level $1$ is $\C[E_4,E_6]$.  As $E_{2k}$ is normalized to have constant term $1$, $E_{10}=E_4 E_6$, and
$$\Delta = \frac{E_4^3-E_6^2}{1728},$$
$E_4\Delta$, $E_6\Delta$, and $E_4^2\Delta$ are the unique normalized cusp forms
of weight $12$, $16$, $18$, and $20$ respectively.  An easy calculation shows
\begin{equation}
\label{e12}
E_{12} = \frac{441 E_4^3 + 250 E_6^2}{691}
\end{equation}
(see, for instance, \cite[\S1.1]{Congruences}).

By a theorem of Serre and Swinnerton-Dyer  $\bar f(q)\in \F_p[[q]]$ is the (mod $p$) reduction of a modular form of level $1$ and weight $k$, then  $q\frac{d \bar f}{d q}$ is the reduction (mod $p$)  of a cusp form of weight $k+p+1$ \cite[\S1.4, Corollaire 2]{Congruences}.
In particular, for $p=13$, we have
$$q\frac{\partial \bar E_4}{\partial q} = 240\bar E_6\bar \Delta$$
and
$$q\frac{d \bar E_6}{d q} = -504\bar E_4^2\bar \Delta.$$
By the Leibniz rule,
$$q\frac{d \bar E_{10}}{d q} = (240\bar E_6^2 -504\bar E_4^3)\bar \Delta.$$
By the von Staudt-Clausen theorem, $\bar E_{12}=1$, which together with (\ref{e12}) implies
$$\bar E_6^2 = 5 + 9 \bar E_4^3,\ \bar\Delta = 8 \bar E_4^3 + 5,\ q\frac{d \bar E_{10}}{d q} = (5\bar E_4^3+4)\bar\Delta.$$
Thus,
$$\frac 14 \overline{2E_{12}+\Delta}^2=\frac 14(2+\bar\Delta)^2 = 1+\bar\Delta + \frac14(8\bar E_4^3+5)\bar\Delta = 1+3q\frac{d \bar E_{10}}{d q},$$
which is multiplicative.  The proposition follows.
\end{proof}

\section{Prime Powers}

A normalized multiplicative power series $\frac 1{2a_0} + \sum_{n=2}^\infty a_n q^n$ is determined by
$a_0$ and the coefficients $a_n$ as $n$ ranges over the set $\P$ of positive
integral powers of primes.  If it is also multiplicative, the  
normalization of $f(q)^2$ is
$$a_0f(q)^2=\frac 1{4a_0}+q+\sum_{n=2}^\infty b_n q^n.$$  
Each $n=p_1^{e_1}\cdots p_k^{e_k}$ which is not in $\P$ determines an equation
$$b_{p_1^{e_1}\cdots p_k^{e_k}}=b_{p_1^{e_1}}\cdots b_{p_k^{e_k}}.$$
Writing $b_i$ formally as a polynomial $B_i(a_0,a_2,a_3,\ldots)$ with integer coefficients
in the variables $a_0$ and $\{a_i\mid i\in\P\}$,
we obtain the polynomial equation
\begin{equation}
\label{poly-P}
P_{p_1^{e_1}\cdots p_k^{e_k}} = B_{p_1^{e_1}\cdots p_k^{e_k}} 
-B_{p_1^{e_1}}\cdots B_{p_k^{e_k}} = 0.
\end{equation}
For $n\ge 2$, let $k$ (resp. $l$) denote the largest element of $\P$ (resp. $\N\setminus \P$) in $[1,n]$, and let
$\X_n$ denote the affine variety 
$$\Spec \C[a_0,a_2,a_3,a_4,a_5,a_7,a_8,a_9, a_{11},\ldots,a_k]/(P_6,P_{10},\ldots,P_l).$$
We identify points on $\X_n$ with polynomials of degree $\le n$ in $\C$.
For $m\ge n$ we have projection morphisms $\phi_{m,n}\colon \X_m\to \X_n$, and for $n\ge 2$, we have the projection 
$\psi_n\colon \X\to \X_n$.

If $f(q)$ is a power series in $q$, we denote by $\E(f)$ the set of $n\ge 2$ such that the $q^n$ coefficient of $f$ is non-zero.

\begin{lem}
\label{Fibers}
If $f, g\in \X_m$ satisfy $\phi_{m,n}(f) = \phi_{m,n}(g)$, then $\E(f-g)\cap [1,2n]$
and
$$\E(f(f-g)) \cap [1,\min(2n,m)]$$ 
are contained in $\P$.  If $\E(f-g)$ contains any element other than $m$, its smallest element satisfies $k,k+1\in\P$.
\end{lem}

\begin{proof}
If $k\le 2n$ is not in $\P$, then $k = k_1 k_2$, where $k_1$ and $k_2$ are relatively prime and $\le n$.  The $q^{k_1}$ and $q^{k_2}$ coefficients
of $f$ and $g$ coincide, so $f,g\in \X_m$ and $k\le m$ implies that the $q^k$ coefficients of $f$ and $g$ are the same, giving the first claim.
As $\E((f-g)^2)\subset [2n+2,\infty)$, we have
$$\E(f(f-g)) \cap [1,\min(2n,m)] = \E(f^2 - g^2) \cap [1,\min(2n,m)].$$
If $k\le 2n$ is not in $\P$, we factor as before, and if $k\le m$, the $k_1 k_2$ coefficients of $f^2$ and $g^2$ are determined by the $k_1$ and $k_2$ coefficients and
are therefore the same.  For the last claim, we note that if $k\le m-1$ is the smallest element of $\E(f-g)$ and $k+1\not\in \P$, then $k+1\not\in \E(f-g)$ and $k+1\not\in \E(f(f-g))$
implies that the $q$ coefficient of $f$ is $0$, which is impossible.
\end{proof}

\begin{cor}
\label{InfinityFibers}
If $k>0$ is the minimal element of $\E(f-g)$ for $f,g\in\X$, then $k,k+1\in\P$.
\end{cor}

\begin{proof}
Without loss of generality, we may assume $k\ge 5$.  The corollary follows by applying Lemma~\ref{Fibers} to $\psi_m(f)$ and $\psi_m(g)$ for $m=k+1$ and $n=k-1$.
\end{proof}

The condition $k,k+1\in\P$ is very strong:

\begin{lem}
\label{Basic}
If $k$ and $k+1$ both belong to $\P$ and $k>8$,
then $k$ is a Mersenne prime or $k+1$ is a Fermat prime.
\end{lem}

\begin{proof}
Either $k$ or $k+1$ is even and therefore a power of $2$. 
The highest power of $2$ dividing $p^{2^r(2s+1)}\pm1$ is at most $2^{r+1}$ 
times the highest power of $2$ dividing $p\pm1$.
Therefore, the only solutions of $2^m-p^n=\pm1$ in integers $m,n,p>1$
is $(3,2,3)$.  If we allow $n=1$ but insist that $p$ is prime,
we obtain precisely the solutions of Mersenne and Fermat type.
\end{proof}

For  use in the next two sections, we prove a number of facts about $\P$ with special reference to Mersenne and Fermat primes.

\begin{lem}
\label{MersenneGaps}
If $p>7$ is a Mersenne
prime, then
$$p+n\notin\P\quad\forall n\in\{2,3,5,7,8,9,11,13,14,15\}.$$
Moreover, either $p+4\notin\P$ or $p+6\notin\P$.
\end{lem}

\begin{proof} Every Mersenne prime is of the form $2^\ell-1$
for $\ell$ prime, and we may assume $\ell>3$.
For $n$ odd between $3$ and $15$, $p+n$ is even and
lies strictly between $2^\ell$ and $2^{\ell+1}$.  For $n=2$,
$p+n\not\in\P$ by Lemma~\ref{Basic}.  For $n\in\{8,14\}$,
$p+n$ cannot be in $\P$ since it is divisible by $3$ but is either $5$ or $7$ (mod $8$) and therefore not a power of $3$.  
Finally, one of 
$p+4$ and $p+6$ is divisible by $7$ but cannot be a power of $7$ since neither $3$ nor $5$ is a power of $7$ (mod $8$).
\end{proof}

\begin{lem}
\label{FermatGaps}
If $p>17$ is a Fermat prime,
then
$$p+n\notin\P\quad\forall 
n\in\{1,3,4,5,7,8,9,10,11\}.
$$
\end{lem}

\begin{proof}
For $n\le 11$ odd, $p+n$ is an even number strictly between two
consecutive powers of 2.  For $n=4$, $p+n$ is divisible by 3
but is not congruent to 3 (mod 8).  If it is a power of 3, it is therefore
a perfect square, which is impossible since 
$$\Bigl(2^{2^{k-1}}\Bigr)^2<2^{2^k}+1+n<\Bigl(2^{2^{k-1}}+1\Bigr)^2.$$
For $n=8$, $p+n$ is not a square and therefore cannot be in $\P$ 
by a (mod 40) argument.  Finally, $p+10$ is divisible by $3$.  If $p+10=3^r$, 
the congruences 
$\con{3^r}{12}{17}$ and $\con{3^r}{12}{257}$ would imply the inconsistent 
congruences $\con{r}{13}{16}$ and $\con{r}{97}{256}$.  This rules out
the case $p>257$, and for $p=257$, $267\not\in \P$.
\end{proof}

\begin{lem}
\label{DoubleF}
If $p>5$ is a Fermat prime, then $2p\pm 1\notin\P$.
\end{lem}

\begin{proof}
As $p=2^{2^k}+1$ and $k\ge 2$, we have $3\vert 2p-1$ and $5\vert 2p+1$.
By Lemma~\ref{Basic}, $2p-1$ cannot be a power of $3$.  As for $2p+1$, it is
congruent to $3$ (mod $8$), so it cannot be a power of $5$.
\end{proof}

\begin{lem}
\label{TripleM}
If $p>7$ is a Mersenne prime and $2p+3\in\P$, then $3p+4\not\in\P$.
\end{lem}

\begin{proof}
Assuming $p=2^n-1$ is Mersenne, if $2p+3$ is prime, it is a Fermat prime $\ge 5$ and therefore $2$ (mod $5$).
It follows that $\con p25$, and therefore that $5$ divides $3p+4$.  If $3p+4$ is of the form $5^m$, then 
$\con{5^m}1{2^n}$.  The highest power of $2$ dividing $5^m-1$ is $2^{m+2}$,
$$3 \ge \frac{5^m-1}{2^{m+2}} > 2^{m-2}.$$
This implies $m\le 3$.  Now, $5^3-1$ is not divisible by $3$ at all, and while $(5^2-4)/3$ is a Mersenne prime, it is not greater than $7$.
\end{proof}

\begin{lem}
\label{MandM}
If $p_1,p_2 > 3$ are  Mersenne primes, then $p_1+p_2+1\not\in\P$.
\end{lem}

\begin{proof}
Mersenne primes greater than $3$ are always congruent to $1$ (mod $3$).  Thus, $3$ divides $p_1+p_2+1$.
However, $3^n+1$  is never divisible by $8$, to $p_1+p_2+1$ cannot be a power of $3$.
\end{proof}

\begin{lem}
\label{MandF}
If $3<p_1 < p_2$, $p_1$ is a Mersenne prime, and $p_2$ is a Fermat prime, then $2p_1+p_2+2\not\in\P$.
\end{lem}

\begin{proof}
As $p_1$ and $p_2$ are Mersenne and Fermat respectively, they are $1$ and $2$ (mod $3$) respectively, so $2p_1+p_2+2$ is divisible by $3$.  As $\con{2p_1+p_2+2}18$,
if $2p_1+p_2+2\in \P$, there exists $n$ such that
$2p_1+p_2+2=3^{2n}$.  If $n = 2^r s$, where $s$ is odd, then the highest power of $2$ dividing
$3^n-1$ is $2^{r+3}$.  If $p_1 = 2^m-1$, then $r+3 \ge m+1$.   As $p_2-1$ is a perfect square, it is less than or equal to $(3^n-1)^2$, so
$$2\cdot 3^n-1 \le 2p_1+3 = 2^{m+1}+1\le 2^{r+3}+1,$$
i.e.,
$$3^{s\cdot 2^r} \le 3^n \le 2^{r+2}+1 < 3^{r+2}.$$
Thus, $s\cdot 2^r < r+2$, so $s=1$ and $r\in \{0,1\}$.  This is impossible since $p_1\ge 7$.
\end{proof}

\begin{lem}
\label{Cantor}
For every odd prime $\ell$, every positive integer
$d$ not divisible by $\ell$, and every residue class (mod $d$), there exists
an integer $n\le (2d^2)^{9\log (2d)}$ such that $n$ belongs to the 
specified residue class and $2n\choose n$ is not divisible by $\ell$.
\end{lem}

\begin{proof} Let $\ell=2k+1$.  If the digits in the base $\ell$ expansion of
$n$ are all $\le k$, then the second condition is satisfied.  
In particular, if $k\ge d$, then the theorem is certainly true since
$n$ can then be chosen to be the representative of the residue class of $a$ in
$[0,d-1]$.  We therefore assume that $\ell<2d$.  

Let 
$$F_r(x)=\prod_{i=0}^{r-1}
\Bigl(1+x^{\ell^i}+x^{2\ell^i}+\cdots+x^{k\ell^i}\Bigr).$$
Then $F_r(x)$ is a sum of distinct terms $x^n$ where $\ell\nmid{2n\choose n}$.
We would like to show that for a suitable value of $r$,
all residue classes of $p$ are represented among the exponents of $F_r(x)$.
As $F_r(1)=(k+1)^r$, it suffices to prove that 
$$|F_r(\zeta^i)|<\frac{(k+1)^r}{d},\ \zeta=e^{2\pi i/d},\ 1\le i < d.$$  

If $m$ is not congruent (mod $d$) to an integer in the interval 
$[-3d/4\ell,3d/4\ell]$, then 
$$\Bigm|\sum_{j=0}^k \zeta^{jm}\Bigm|=\frac{\bigm|1-\zeta^[(k+1)m\bigm|}
{\bigm|1-\zeta^m\bigm|}\le\frac{2}{|2\sin(\pi m/d)|}\le
\frac1{|\sin(3\pi/4\ell)|}.$$
For $0\le x\le \pi/6$, $\sin x\ge 3x/\pi$.  Therefore, if $\ell>3$, 
then 
$$\Bigm|\sum_{j=0}^k \zeta^{jm}\Bigm|\le\frac{4\ell}9<\frac{8(k+1)}9.$$
On the other hand, if $\ell=3$, then $|1+\zeta^m|\le \sqrt2<\frac{8(k+1)}9$.

If $m$ is not congruent to 0 (mod $p$) and $\ell^s>p$, then  
$m,\,m\ell,\,m\ell^2,\,\ldots,\,m\ell^{s-1}$ cannot all be congruent 
(mod $d$) to integers in $[-3d/4\ell,3d/4\ell]$.  Therefore, the product 
of any $s$ consecutive multiplicands in $F_r(\zeta)$ is less than 
$\frac89(k+1)^s$.  If $t>{\log d\over\log 9-\log 8}$, then 
$$|F_{st}(\zeta^m)|<F_{st}(1)/d.$$
We may therefore take $n$ to be less than 
$$\ell^{st}<\ell^{\LP\frac{\log d}{\log \ell}+1\RP t}<(2d^2)^{9\log (2d)}.$$
The proposition follows.\end{proof}

\begin{lem}
\label{nondivisor}
For all integers $k>1$, there exists
a prime $\ell\le 4k+1$ such that $\ell$ does not divide $2^k-2$.
\end{lem}

\begin{proof} For any $s\in\N$, 
$${2^{2s}\over 2s}={2+\sum_{i=1}^{2s-1}{2s\choose i}\over2s}
\le\binom{2s}s=\frac{(2s)!}{s!^2}=\prod_p p^{k_p},$$
where $p^{k_p}\le 2s$ for all $p$.  As $\pi(n)\le \frac{n+1}2$ and 
$\prod_{p\le n} p\ge n$, 
$$\prod_{p\le 2s}p\ge\frac{2^{2s}}{2s\prod_{p\le\sqrt{2s}} p^{k_p-1}}
\ge \frac{2^{2s}}{2s\prod_{p\le\sqrt{2s}} {2s\over p}}
\ge 2^{2s}(2s)^{-1-\sqrt{2s}/2}.$$  

For $s\ge 16$, we have $1-\sqrt{2}/2<s^{-1/2}$ and $\log 
2s<\frac 32\sqrt{s}\log 2$, so 
$$\sum_{p\le 2s}\log p\ge 2s\log 2-\LP{\sqrt{2s}\over 2}+1\RP\log 2s
\ge 2s\log2-\sqrt{s}\log 2s\ge{s\log2\over2}.$$  
If $\ell$ is the smallest prime not dividing $2^k-2$ and $\ell>31$ then 
$s={\ell-1\over2}\ge16$, so 
$${s\log2\over2}\le\sum_{p\le2s}\log p\le\log(2^k-2)<k\log2.$$
Thus, $\ell\le 4k+1$.  This proves the existence of the desired prime 
$\ell$ when $k\ge8$.  For $k\le 7$, we can set $\ell=5$ except for $k=5$, for
which we can set $\ell=7$. 
\end{proof}

\begin{lem}
\label{AP}
For $k\ge5$, an arithmetic progression of integers with initial term,
$a\in [1,2^{2k+1}]$, common difference $2^k-2$, and length
$2^{k-2}$ contains an integer not in $\P$, except when $(k,a)=(5,19)$.
\end{lem}

\begin{proof} For $k\ge 6$, there exists a prime $\ell<2^{k-3}$ such that 
$\ell\nmid 2^k-2$.
Then any such progression contains at least two terms divisible by $\ell$,
differing by $(2^k-2)\ell$.  At least one is not divisible by $\ell^2$, 
so if they are both powers of $\ell$, then $2^k-1=\ell^{r-1}$.  By Lemma~\ref{Basic},
this means $r=2$ which is impossible since $\ell\le 2^{k-3}$.  Thus, 
the progression contains an integer not in $\P$.  For $k=5$, it is easy
to check that $a=19$ is the only initial term which gives an $8$-term 
progression consisting only of elements of $\P$.
\end{proof}

\section{Reduction to the Sparse Case}

The polynomial equations $P_n$, $n\not\in \P$ are of weighted degree
$n$ where each variable $x_m$ has degree $m$.  They are therefore linear
in $x_m$ for $m>n/2$.  In this section we systematically exploit this 
linearity.

\begin{prop}
\label{Matrices}
Let $n\le 15$ be an integer.  Let
$$
F=\begin{pmatrix}
a_2&a_4&a_5&a_6&a_8&a_9&a_{10}&a_{11}&a_{12}\\
1&a_3&a_4&a_5&a_7&a_8&a_9&a_{10}&a_{11}\\
0&1&a_2&a_3&a_5&a_6&a_7&a_8&a_9\\
0&0&0&0&1&a_2&a_3&a_4&a_5\\
\end{pmatrix},
$$
$$M'=\begin{pmatrix}
a_2&a_3&a_4&a_5&a_7&a_8&a_9&a_{11}&a_{13}&a_{14}\\
1&a_2&a_3&a_4&a_6&a_7&a_8&a_{10}&a_{12}&a_{13}\\
0&0&0&0&1&a_2&a_3&a_5&a_7&a_8\\
0&0&0&0&0&0&0&1&a_3&a_4\\
0&0&0&0&0&0&0&0&1&a_2\\
\end{pmatrix},
$$
$$M''=\begin{pmatrix}
a_2&a_3&a_5&a_6&a_7&a_8&a_9&a_{11}&a_{13}&a_{14}&a_{15}\\
1&a_2&a_4&a_5&a_6&a_7&a_8&a_{10}&a_{12}&a_{13}&a_{14}\\
0&0&1&a_2&a_3&a_4&a_5&a_7&a_9&a_{10}&a_{11}\\
0&0&0&0&0&0&0&1&a_3&a_4&a_5\\
0&0&0&0&0&0&0&0&1&a_2&a_3\\
\end{pmatrix}.
$$
If every point in 
\begin{equation}
\label{rank}
\{f\in\X_n\mid \rk(F)\le 3\}\cup\{f\in\X_n\mid\rk(M')\le 4\}\cup
\{f\in\X_n\mid\rk(M'')\le 4\}
\end{equation}
is the image of one and only one point of $\X$, then
$\X$ is of affine type.
\end{prop}

\begin{proof} We apply Proposition~\ref{AffineType}, where $\V\subset \X_n$ is the union of the three closed subvarieties defined
by the conditions that $F$, $M'$, or $M''$ is not of full rank.  The hypothesis guarantees condition (2) and condition (1) for elements of
$\X_n$ in which at least one of the matrices is not of full rank.  In verifying the remaining conditions, we may therefore assume that
all three matrices are of full rank.

Suppose $f,g\in \X$ map to the same element in $\X_n$, and let $k=\inf \E(f-g)$.  By Corollary~\ref{InfinityFibers}
and Lemma~\ref{Basic}, either $k+1$ is a Fermat prime or $k$ is a  Mersenne prime.

Suppose $k+1$ is Fermat.  By Lemma~\ref{Fibers} and Lemma~\ref{FermatGaps}, 
$$\E(f(f-g))\cap [k,k+12] \subseteq \{k,k+1,k+3,k+7\}.$$
Defining $x_i$ to be the $q^{k+i}$ coefficient of
$f-g$ for $i=0,1,\ldots,12$, we have $x_i=0$ for $i\in\{2,4,5,6,8,9,10,11,12\}$, so we obtain
$$(x_0\,x_1\,x_3\,x_7)F=0,$$
which by the rank condition implies that $x_0=x_1=x_3=x_7=0$, contrary to the definition of $k$.

If $k$ is Mersenne, then by Lemma~\ref{Fibers} and Lemma~\ref{MersenneGaps}, either
$$\E(f(f-g))\cap [k,k+15] \subseteq \{k,k+1,k+4,k+10,k+12\}$$
or
$$\E(f(f-g))\cap [k,k+15] \subseteq \{k,k+1,k+6,k+10,k+12\}.$$
Defining $x_i$ to be the $q^{k+i}$ coefficient of $f-g$ for $i=0,1,\ldots,15$, we have $x_i=0$ for $i\in\{2,3,4,5,7,8,9,11,13,14,15\}$
or $i\in\{2,3,5,6,7,8,9,11,13,14,15\}$ respectively, and we therefore have
$$(x_0\,x_1\,x_4\,x_{10}\,x_{12})M''= 0$$
or
$$(x_0\,x_1\,x_6\,x_{10}\,x_{12})M'= 0$$
respectively.  Either way, we get a contradiction, implying that $f=g$, as claimed.   This gives condition (1) for $y\in \X_n\setminus \V$.

A slight variant gives (3).  In the Fermat case, we assume $m\ge n+12$, and let $f,g\in \X_m$ map to the same element in $\X_n$ but different elements in $\X_{n+1}$.
By Lemma~\ref{Fibers}
and Lemma~\ref{Basic}, either $k+1$ is a Fermat prime or $k$ is a Mersenne prime, and the argument proceeds as before.  In the Mersenne case, we assume $m\ge n+15$,
and otherwise the argument is the same.

Now consider a bounded open neighborhood $U\subseteq \X_n$ of polynomials such that the full rank condition for $F$, $M'$, and $M''$  and the condition $a_0\neq 0$ hold on 
the closure $\bar U$ in the complex topology.  For (4), it is enough to show that  for each such $U$
there exists $m>n$ such that $\phi_{m,n}^{-1}(U)$ is bounded.  If $n+1\not\in\P$, then factoring $n+1=k_1k_2$, where the $k_i>1$ are relatively prime,
we can take $m=n+1$, since $a_{n+1} =a_{k_1}a_{k_2}$ is bounded on $U$.  If $n+2\not\in\P$, we can take $m=n+2$.  Factoring $n+2=k_1k_2$, $a_{n+2}=a_{k_1}a_{k_2}$ is bounded on $U$, and the same is true 
for $a_{n+1}$, since
$$a_{k_1}a_{k_2}+2a_0 a_{n+1} + a_0 (a_2 a_n + a_3 a_{n-1} + \cdots + a_n a_2) = b_{n+2} = b_{k_1}b_{k_2}$$
can be regarded as a linear equation in $a_{n+1}$ whose coefficients are polynomial in $a_0,a_2,\ldots,a_n$.  Thus, it suffices to consider the cases that $n$ is Mersenne or that $n+1$ is Fermat.

If $n+1$ is Fermat, we take $m=n+12$.  Each of $n+2, n+4, n+5, \ldots, n+12$ can be written as a product of two relatively prime integers $\le n$, so $a_{n+2},\ldots,a_{n+12}$ are bounded on $U$.
To prove that $a_n, a_{n+1}, a_{n+3}, a_{n+7}$ are likewise bounded on $U$, we note that $a_k$ and $b_k$ are bounded on $U$ for $k-n\in \{2,4,5,6,8,9,10,11,12\}$, so
$$(a_0\,a_1\,a_3\,a_7)F$$
is bounded on $U$.  As $F$ is of full rank on $\bar U$, this implies $a_0, a_1, a_3, a_7$ are bounded on $U$.  The same argument applies to Mersenne primes, taking $m=n+15$ and using $M'$ or $M''$ in place of $F$.

\end{proof}

\begin{lem}
\label{Rank}
If $\rk(F)\le 3$, $\rk(M')\le 4$, or $\rk(M'')\le 4$, then 
either $a_2=a_3=a_4=0$ or $a_3 = 1$ and $a_0 = a_2=a_4 = \pm 1$. 
\end{lem}

\begin{proof} This follows by solving the equations 
$P_6,\,P_{10},\,P_{12},\,P_{14},\,P_{15}$ of (\ref{poly-P}) together with the equations 
expressing any of the three rank conditions in (\ref{rank}).
\end{proof}

\begin{prop}
\label{Ones}
If $\epsilon\in\{1,-1\}$, $f\in\X$, and 
$$f={\epsilon\over 2}+q+\epsilon q^2+q^3+\epsilon q^4 + a_5 q^5 + a_6 q^6+\cdots,$$
then $a_n=\epsilon^{n+1}$ for all $n\ge1$.
\end{prop}

\begin{proof} By Lemma~\ref{SignChange}, we may assume $\epsilon=1$.  Solving the equations $P_n$, $6\le n\le 
72$, we obtain the unique solution $a_n=1$ for $2\le n\le 15$.  
Thus any solution $f$ maps to the same element of $\X_{15}$ as $\frac 12 + \sum_{i=1}^\infty q^i$, which is a special case of solution (xii).
By Corollary~\ref{InfinityFibers} and Lemma~\ref{Basic}, the smallest value $m$ for which $a_m \neq 1$ is either a Mersenne prime or one less than a Fermat prime.
Either way,  
comparing $f$ with
$$g = \frac12 + (a_m-1)q^m + (a_{m+1}-1)q^{m+1} + \sum_{i=1}^{2m-1} q^i \in \X_{2m-1}$$
either $\psi_{2m-1}(f) = g$ or $k = \inf \E(\psi_{2m-1}(f) - g) $ satisfies $k,k+1\in\P$. 
 
If $m+1$ is a Fermat prime, then this is possible if and only if $k=2m-1$ (in which case $k$ is a Mersenne prime).
Whether $a_{2m-1}=1$ or not, $2m+1,2m+3\not\in\P$ by Lemma~\ref{DoubleF}, so 
the equations $P_{m+2}$, 
$P_{2m+1}$, and $P_{2m+3}$ read:
\begin{align*}
m+2&=2a_m+2a_{m+1}+m-2,\\
2m+1&=2a_m a_{m+1}+2a_{2m-1}+2a_m+2m-5,\\
2m+3&=2a_{2m-1}+2a_m+2m-1.\\
\end{align*}
Solving, we obtain $a_m=a_{m+1}=1$, giving a contradiction.

We may therefore assume that $m>7$ is Mersenne.  In this case, $a_i=1$ for $i\le 2m-1$, and also $m+2,2m\not\in\P$, so $a_{m+2}=1$, $a_{2m}=a_m$,
$b_{m+2}=m+2$, and $b_{2m}=2b_m=2m+2a_m-2$.  Solving $P_{m+2}$ and $P_{2m}$, we obtain $(a_m,a_{m+1}) \in \{(1,1), (2,0)\}$.
By hypothesis, the latter alternative must be true.  We now define 
$$g = \frac12 + \sum_{i=1}^{m^2+m-1} c_iq^i$$
where
$$c_i=\begin{cases}2&\text{if } m\vert i,\\
0 &\text{if } m+1\vert i,\\
1&\text{otherwise.}\\
\end{cases}$$
It is impossible that $g = \psi_{m^2+m-1}(f)$, since $g$ does not lift to an element of $\X_{m^2+m}$.
Let $k= \inf \E(f-g)$.

If $k\le 4m$, then $k=2m+2$,  $k+1$ is Fermat, and $a_k\neq c_k=0$.  The multiplicativity of $f^2$ implies $a_{3m+2} = 2a_{2m+2}+1$ and $a_{3m+4} = -2a_{2m+2}+1$.
Now by Lemma~\ref{TripleM}, $3m+4\not\in\P$.  Therefore, $a_{3m+4}=1$, contrary to assumption.  Thus, we may assume $k\ge 4m$.

We define $c_i$ as above.  If $k$ is Mersenne, we define
$$g = \frac12 + \sum_{i=1}^{k+m+1} c_iq^i + (a_k-1)(q^k-q^{k+1}-q^{m+k-1}).$$
Then $\psi_{m+k-2}(g)\in \X_{m+k-2}$, so it coincides with $\psi_{m+k-2}(f)$  If $m+k-1\not\in\P$, then $P_{m+k-1}$ shows there is no way of lifting $\psi_{m+k-2}(f)$
to $\X_{m+k-1}$, which is absurd.  Thus, $m+k-1\in\P$, and $\psi_{m+k}(g)\in \X_{k+m}$ must coincide with $\psi_{m+k}(f)$.  By Lemma~\ref{MandM}, $k+m+1\not\in\P$, so
$P_{m+k+1}$ shows there is no way of lifting $\psi_{m+k}(f)$ to $\X_{m+k+1}$, which is absurd.

If $k+1$ is Fermat, we define
$$g = \frac12 + \sum_{i=1}^{k+2m+3} c_iq^i + a_k(q^k(1-q)+2q^{m+k}(1-q^2)-3q^{2m+k-1}(1-q^2)^2).
$$
Now, $\psi_{m+k-1}(g) \in \X_{m+k-1}$ coincides with $\psi_{m+k-1}(f)$.  If $m+k\not\in\P$, then $P_{m+k}$ shows there is no way of lifting $\pi_{m+k-1}(f)$ to $\X_{m+k}$,
which is absurd.  We repeat the argument, replacing $m+k-1$ successively by $m+k+1$, $m+2k-2$, $m+2k$, $m+2k+2$.  We conclude that $m+2k+2\in\P$, which is impossible
by Lemma~\ref{MandF}.

The only remaining possibility is $k = m^2+m-1$, and $k\in \P$.  In this case, we set
$$g = \frac12 + \sum_{i=1}^{m^2+2m-2} c_iq^i$$
where
$$c_i=\begin{cases}
2&\text{if }i=m^2+m-1,\\
2&\text{if } i\le m^2\text{ and } m\vert i ,\\
0&\text{if }i=m^2+2m-2 \text{ or } m+1\vert i,\\
1&\text{otherwise.}\\
\end{cases}$$
It is impossible that $m^2+m\le \E(f-g) \le m^2+2m-2$ since there is no Mersenne or Fermat prime in that interval.  However,
$g$ does not lift to an element of $\X_{m^2+2m-1}$, which gives a contradiction.
\end{proof}

\begin{lem}
\label{Sp}
Suppose $a_2=0$ and $\rk(F)\le 3$, $\rk(M')\le 4$, or 
$\rk(M'')\le 4$.  Let 
\begin{equation}
\label{sparsity}
m=\inf\{i\ge2\mid a_i\neq 0\}
\end{equation}
either $m$ is undefined (in which case $f(q)$ is linear in $q$), $m$ is a
Mersenne prime $\ge 31$, or $m+1$ is a Fermat prime $\ge 257$.
\end{lem}

\begin{proof} As $a_2=0$, we have $a_3=a_4=0$, and also $a_6=0$.  Equation $P_6$
implies $a_5=0$, and $P_{15}$ implies $a_7 a_8=0$.  If $a_7\neq 0$, the
equations $P_i$ as $i$ runs through all positive integers $\le 92$
not in $\P$ are inconsistent; if $a_7=0$, the equations up through
$i=34$ imply $a_8=a_9=\cdots =a_{16}=0$.  The multiplicativity of $f$
implies $m\in\P$; the multiplicativity of $f^2$ implies $m+1\in\P$.  
The result now follows from Lemma~\ref{Basic}.\end{proof}

\begin{defn}
We say $f$ is \emph{sparse} if the \emph{index} $m$
of (\ref{sparsity}) is $\ge 16$.
\end{defn}

\begin{prop}
\label{ExceptSparse}
If $f,g\in\X$ are not sparse and 
$f\equiv g\ (\hbox{mod }x^{17})$, then $f=g$.  In other words, a non-sparse
element of $\X$ is determined by its first 17 coefficients.
\end{prop}

\begin{proof} By Lemma~\ref{Rank} and Proposition~\ref{Ones}, if $F$, $M'$, or $M''$ has less than full rank
then either $f$ is a solution of type (xi) or $a_2=0$.  In the latter case,
$f$ is sparse by Lemma~\ref{Sp}.  Thus, we may assume full rank.  By Proposition~\ref{Matrices}, this 
implies that all higher coefficients are determined from the first 17 
coefficients.  \end{proof}



\goodbreak

\section{Sparse Solutions and Mandelbrot Polynomials}

\begin{lem}
\label{six}
If $f\in\X$ is sparse, then $a_n=0$ except when 
\con{n}{1}{6}.
\end{lem}

\begin{proof} Let $n$ be the smallest positive integer not of the form $6k+1$ for which 
$a_n\neq 0$.  Thus $b_m=0$ when $0<m<n$ and $m$ is congruent to 3, 4, 
5, or 0 (mod 6).  Therefore $n,n+1\in \P$.  As $n\ge 16$, either $n$ is 
Mersenne or $n+1$ is Fermat.  The first is impossible since every
Mersenne prime greater than $3$ is $1$ (mod $6$).  Thus $n+1$ is a Fermat
prime.  As $\con{n+2}{0}{6}$, we have $b_{n+2}=0$ and therefore
$a_{n+1}=0$.  As $\con{2n+2}{4}{6}$, $b_{2n+2}={a_{n+1}^2\over 2a_0}=0$.
Thus, $b_{n+1}=0$, which is impossible since $a_n\neq 0$ and $\con{n+1}{3}{6}$.
\end{proof}

\begin{cor}
\label{MersenneIndex}
If $f$ is a nonlinear sparse power series,
its index is a Mersenne prime.
\end{cor}

\begin{lem}
\label{Cong}
If $f\in\X$ is a sparse series of index $p$ then for all 
$n>1$ with $a_n\neq0$, there exists $r$ such that $n\equiv r\hbox{ (mod 
$p-1$)}$, and $1\le r\le {2n\over p-1}-1$.
\end{lem}

\begin{proof} We proceed by induction.  If 
$$n_i\equiv r_i\hbox{ (mod $p-1$)},\ r_i\le{2n_i\over 
p-1}-1\quad(i=1,\,2)$$
then 
$$n_1 n_2\equiv r_1 r_2\hbox{ (mod $p-1$)}$$
and
$$1\le r_1 r_2\le 
{4\over(p-1)^2}n_1 n_2-\LP {2n_1+2n_2\over p-1}-1\RP\le
{2\over p-1}n_1 n_2-1.$$
It therefore suffices to consider the case $n\in\P$.  The induction 
hypothesis certainly implies that whenever $k<n$ and $b_k\neq 0$, there 
exists $\con{r}{k}{p-1}$, $1\le r\le {2\over p-1}(n-1)$.  If the 
residue of $n$ does 
not belong to any residue class in $[1,2n/(p-1)-1]$, neither does the 
residue of $n+1$ since a prime power cannot be congruent to 0 modulo 
$p-1\not\in\P$.  Therefore $b_{n+1} = 2a_n\neq 0$, and this implies 
$n+1\in\P$.  
By Corollary~\ref{MersenneIndex}, $n$ is a Mersenne prime.  Thus $p+1\mid n+1$, and 
$$\con{n+1}{2{n+1\over p+1}}{p-1}.$$
As 
$$2{n+1\over p+1}-1\le 2{n-1\over p-1}-1\le{2n\over p-1}-1,$$
the lemma follows by induction.\end{proof}

\begin{prop}
\label{notone}
If $f\in\X$ is sparse of index $p$, then $a_p\neq 
1$.
\end{prop}

\begin{proof} If $a_p=1$, then $P_{2p}$ implies $a_{2p-1}=0$.  By Lemma~\ref{Cong}, if 
$a_i$ and $a_j$ are non-zero, $\con{i+j}2{p-1}$, and $i+j<p^2/2$, then either 
$i=0$, $j=0$, or $\con{i\equiv j}1{p-1}$.  The first two possibilities 
are ruled out by Lemma~\ref{six} (note that $p\equiv 1\pmod6$), so $p-1$ must divide $i-1$ and $j-1$.  Thus
$$b_{k(p-1)+2}={a_{k(p-1)+1}\over a_0}$$
for $1\le k < {p+1\over 2}$.  For $2\le k<{p+1\over 2}$, the highest
power of $2$ dividing $k(p-1)+2$ is less than $p+1$, so $a_{k(p-1)+1}=0$.

Equation $P_{p(p+1)}$ guarantees that there is some $n>p$
for which $a_n\neq0$.
Suppose that the smallest such $n$ satisfies $n < p^2/2$.
We have just proved that $n\not\equiv1$ (mod $p-1$).
For $r,s>1$, $a_r 
a_s\neq 0$ implies $rs\ge p^2$;  thus $n\in\P$.  Likewise, $b_{n+1} = 2 a_0 
a_n \neq 0$ so $n+1\in\P$.  By Lemma~\ref{six}, $n+1$ cannot be a Fermat prime, so $n$ is a Mersenne prime.  
If $n$ reduces to $s$ (mod $p-1$), $1<s<p-1$, an easy induction shows 
that for every residue class $r$, $1<r<s$, and every $m<p^2/2$, 
$\con{m}{r}{p-1}$, we have $a_m=0$.  Therefore, 
$$0=b_{n+p}={a_{n+p-1}+a_n\over a_0},\ 0=b_{n+2p-1}={a_{n+2p-1}+a_{n+p-1}\over
  a_0},\ \ldots$$
In particular, if $n\le m < p^2/2$, and $\con mn{p-1}$, then $a_m\neq 
0$.  Since the largest possible Mersenne prime less than $p^2/2$ is 
$\le\frac{(p+1)^2}{4}$, we have an arithmetic progression of at least 
$(p+1)/4$ terms with common difference $p-1$ and every term in $\P$.  

If the smallest element $n$ of the set $\{n>p\mid a_n\neq 0\}$
exceeds $p^2/2$, then proceeding 
as before, either $n$ is Mersenne (necessarily $\frac{p^2+2p-1}2$) or 
$n=p^2+p-1$.  In either case, by induction
$$a_n=-a_{n+(p-1)}=a_{n+2(p-1)}=-a_{n+3(p-1)}=\cdots=-a_{n+\frac{p-3}4(p-1)}.$$
Thus, the proposition follows from Lemma~\ref{AP}.
\end{proof}

\begin{defn}
We define the \emph{Mandelbrot polynomials} $M_i(y)$
by the recursive formula 
$$M_n(y)=\begin{cases}y&\text{if $n=1$,}\\
-\frac12\sum_{i=1}^{n-1} M_i(y) M_{n-i}(y)&\text{if  $n>1$  odd},\\
\frac12 M_{n/2}(y)-\frac12\sum_{i=1}^{n-1} M_i(y) M_{n-i}(y)&\text{if 
$n>1$ even}.\\
\end{cases}$$
\end{defn}
Thus, 
$$\begin{gathered}
M_2(y)=\frac{-y^2+y}2,\ M_3(y)=\frac{y^3-y^2}2, \\
M_4(y)=\frac{-5y^4+6y^3-3y^2+2y}8,
\ M_5(y)=\frac{7y^5-10y^4+5y^3-2y^2}8,\\
M_6(y)=\frac{-21y^6+35y^5-21y^4+13y^3-6y^2}{16},\ldots
\end{gathered}$$

The definition is motivated by the following proposition:

\begin{prop}
\label{multizero}
If $f\in\X$ is sparse of index $p$, then
for $1\le i\le{p-1\over2}$, $M_i(a_p)=0$ whenever $i(p-1)+1\not\in\P$.
\end{prop}

\begin{proof} Let $c_i=a_{i(p-1)+1}$.  We have seen that
$$b_{k(p-1)+2}={1\over2a_0}\sum_{i=0}^k c_i c_{k-i}.$$
For $2\le k \le {p-1\over2}$, $k(p-1)+2$ is even, but the highest
power of 2 dividing it is $<p+1$.  Therefore,
$$b_{k(p-1)+2}
=\begin{cases}0&\text{if $k\ge 3$ is odd,}\\
\frac{b_2 b_{k(p-1)+2}}2 &\text{if $k\ge 2$ is even.}\\
\end{cases}$$
As $b_2=1/2a_0$, we have 
$$\sum_{i=0}^k c_i c_{k-i}
=\begin{cases}0&\text{if  $k\ge 3$ is odd,}\\
\frac{b_{k(p-1)+2}}2=c_{k/2}&\text{if  $k\ge 2$ is even.} \\
\end{cases}$$
As $c_0=1$ and $c_1=a_p$, the proposition follows by induction.\end{proof}

\begin{cor}
\label{TwoThousand}
If $f\in\X$ is sparse, 
then its index must be greater than $2^{2000}$.
\end{cor}

\begin{proof} No two polynomials $M_i(y)$ for $i\le 11$ have a common root
other than $0$ and $1$.  By machine computation, for every prime $q<2000$
there exist positive integers $i<j\le 11$ such that $i(2^q-2)+1,j(2^q-2)+1\not\in\P$.  (In every case,
this can be witnessed by a prime divisor less than $1000$.)
\end{proof}

\begin{prop}
\label{Integral}
The roots of $M_n(y)$ are always $2$-adically integral.
Moreover, if $p$ does not divide the binomial coefficient $2n\choose 
n$, then the roots are $p$-adically integral.
\end{prop}

\begin{proof} Let $v$ denote the valuation on $\bar\Q_2$ normalized so that $v(2)=1$, and let
$\gamma$ be an element of $\bar\Q_2$ with $v(\gamma)<0$.  We claim that 
for any sequence $\gamma_i$ with $\gamma_1=\gamma$, such that 
$$\beta_n=2\gamma_n+\sum_{i=1}^{n-1}\gamma_i\gamma_{n-i}$$
is zero when $n$ is odd and has valuation at least $v(\gamma_{n/2})$ 
when $n$ is even, the valuation of $\gamma_n$ does not depend on the 
$\beta_i$.  In fact, if $\gamma_i$ and $\delta_i$ are two such 
sequences, then $v(\delta_n-\gamma_n)>v(\gamma_n)=v(\delta_n)$ for all $n$.  

It suffices to treat the case that 
$$2\delta_n+\sum_{i=1}^{n-1}\delta_i\delta_{n-i}=0$$
for all $n\ge 2$, i.e., 
$$\Bigl(1+\sum_{i=1}^\infty \delta_i x^i \Bigr)^2=1+2\gamma x.$$
By the binomial theorem
$$\delta_n=(-1)^{n-1}\frac{(2n-3)!!}{n!}\gamma^n,$$
where, as usual, $k!!$ is the product of all odd numbers up to $k$.  In 
other words, if $d(n)$ denotes the sum of the digits in the binary 
expansion of $n$, 
$$v(\delta_n)=n v(\gamma)+d(n)-n.$$
As $d(i+j)\le d(i)+d(j)$ 
%
$$v(\delta_i\delta_{n-i})\ge v(\delta_n)$$
for $0<i<n$, and if $n$ is even,
$$v(\delta_{n/2}^2) \ge v(2\delta_n).$$

We prove by induction that $v(\gamma_i-\delta_i) > v(\gamma_i) = v(\delta_i)$ for all $i$.
Assume it holds all $i<n$.   Defining $\gamma_{n/2}=\delta_{n/2}=0$ if $n$ is odd,
\begin{equation}
\label{dmg}
\begin{split}
\delta_n-\gamma_n 
= - \frac 12 \beta_n &- \sum_{1\le i < n/2} (\delta_i-\gamma_i) \gamma_j - \sum_{1\le i < n/2} \delta_i (\delta_j-\gamma_j) \\
& - \frac 12(\delta_{n/2}-\gamma_{n/2})(\delta_{n/2}+\gamma_{n/2}).
\end{split}
\end{equation}

If $n$ is even, 
$$v(\beta_n/2)\ge v(\gamma_{n/2}/2)=v(\delta_{n/2}/2)\ge v(\delta_n)-v(\delta_{n/2})>v(\delta_n)$$
since $v(\delta_i) < 0$ for all $i\ge 1$.  The right hand side of (\ref{dmg}) is therefore a sum of terms with valuation strictly larger than $v(\delta_n)$, as claimed.

For $p$ odd, we note that by induction $M_n(y)\in\Z_p[y]$ for all $n$.  
The leading coefficient of $M_n(y)$ is again $(-1)^{n-1}(2n-3)!!/n!$ and is 
therefore not divisible by $p$ if $2n\choose n$ is not.  
\end{proof}

\begin{prop}
\label{integer}
If $f$ is sparse of 
index $m$, then $a_m$ is an algebraic integer.
\end{prop}

\begin{proof} By Corollary~\ref{MersenneIndex}, $m=2^k-1$, and 
by Corollary~\ref{TwoThousand}, we may assume
$k>2000$.  Let $\ell$ be any odd prime and let $p$ be a prime not dividing $m-1$.
By Lemma~\ref{nondivisor},
we may take $p\le 4k+1$; if $k<27720$, we take $p=29$.  An integer
congruent to $-p$ (mod $p^2$) cannot be in $\P$, so we apply Lemma~\ref{Cantor}, with
$d=p^2$, and a residue class $a$ such that $\con{a(m-1)+1}{-p}{p^2}$.
If $2000<k<27720$ and $d=841$ or $k\ge 27720$ and $d\le (4k+1)^2$, then
$m/2>(2d^2)^{9\log (2d)}$, so there exists $n<m/2$ such that 
$n(m-1)+1\not\in\P$ and $\ell\nmid{2n\choose n}$.  By Proposition~\ref{multizero},
$M_n(a_p)=0$, but by Lemma~\ref{Integral}, all the roots of $M_n$ are $\ell$-adically
integral.
\end{proof}  
\medskip

We next consider the generating function 
$$g_c(z)=1+\sum_{n=1}^\infty M_n(c)z^n.$$
By construction, $g_c(z)$ satisfies the formal functional equation 
$$g_c(z)^2=g_c(z^2)+2cz.$$
This motivates the recursive definition 
\begin{equation}
\label{gnc}
g_{n,c}(z)=
\begin{cases}1 & \text{if } n=0, \\ 
                    \sqrt{g_{n-1,c}(z^2)+2cz} &\text{if } n>0. \\
\end{cases}                
\end{equation}
Explicitly, 
\begin{align*}
g_{1,c}(z)&=\sqrt{1+2cz}, \\
g_{2,c}(z)&=\sqrt{\sqrt{1+2cz^2}+2cz}, \\
g_{3,c}(z)&=\sqrt{\sqrt{\sqrt{1+2cz^4}+2cz^2}+2cz},
\end{align*}
and so forth.
This sequence of power series in $z$ converges coefficientwise to $g_c(z)$;
in fact the first $2^n$ coefficients of $g_{n,c}(z)$ coincide with those
of $g_c(z)$.  

So far, we have regarded $g_c(z)$ as a formal power series 
parametrized
by $c$, but each series $g_{n,c}(z)$ converges, for each value $c$, in a disk around
0.  To find the radius of this disk we define recursively
$$I_n(c)=\begin{cases} 0 & \text{if } n=0,\\
                I_{n-1}(c)^2+c & \text{if } n>0. \\
\end{cases}$$
The algebraic function $g_{n,c}(z)$ can have branch points only for $z$
in the set
$$\{z\mid g_{n,c}(z)=0\}\cup \{z\mid g_{n-1,c}(z^2)=0\}\cup \cdots
\cup  
\{z\mid g_{1,c}(z^{2^{n-1}})=0\}.$$
Now, $g_{n,c}(z)$ is integral of degree $2^n$ over $\C[z]$, and its norm
is 
$$I_n(-2c)z^{2^{n-1}}-1.$$
Therefore the power series for $g_{n,c}(z)$
converges in an open disk around $0$ of radius 
$$R_{n,c}=\inf_{1\le k\le n} |I_k(-2c)|^{-2^{1-n}}=
\Bigl(\sup_{1\le k\le n}|I_k(-2c)|\Bigr)^{-2^{1-n}}.$$

\begin{lem}
\label{single}
Let $U$ be a connected neighborhood of $\infty$
in the Riemann sphere such that for all finite $c\in U$, the absolute value
of $I_n(-2c)$ is strictly greater than the absolute values $|I_k(-2c)|$ for
$k<n$.  Then for each $c\in U$, $g_{n,c}(z)^2$ has exactly one zero, denoted $z_{n,c}$, in the disk 
$|z|<R_{n-1,c}^{1/2}$.  Moreover, 
$$z_{n,c}^{-2^{n-1}}=I_n(-2c),$$
and the zero at $z_{n,c}$ is simple.
\end{lem}

\begin{proof}  
First we observe that $g_{n,c}(z)^2=g_{n-1,c}(z^2)+2cz$ is really defined
in the disk $|z|<R_{n-1,c}^{1/2}$.  In particular it is defined at
every $2^{n-1}$st root of $I_n(-2c)^{-1}$.  Since the product
of $g_{n,c}(z)^2$ and its conjugates over the field of rational functions
has only simple zeroes, we need only show that $g_{n,c}(z)^2$ itself 
accounts for exactly one of those zeroes.  We prove this by analytic
continuation, using the fact that in a continuously varying family of
analytic functions, the number of zeroes inside a continuously varying
disk never changes as
long as there is never a zero on the boundary of the disk.  
As $U$ is connected,
it suffices to prove the claim when $|c|\gg 0$.  But in this case it is
clear that each conjugate of $g_{n-1,c}(z^2)+2cz$ accounts for exactly
one of the $2^{n-1}$ roots in question, each according to the constant term
in its power series expansion, which is a different $2^{n-1}$st root of
unity for each conjugate.  
\end{proof}

\begin{lem}
\label{convergence}
If $c\in \C$, $r>0$, and $n\in \N$ are such that 
$g_{n,c}(z)$, $g_{n+1,c}(z)$, $g_{n+2,c}(z)$, $\ldots$ all have radius of
convergence greater than $r<1$, then $g_c(z)$ has radius of convergence 
greater than $r$ and the sequence $\{g_{k,c}(z)\}_{k\ge n}$ converges
to $g_c(z)$ on the closed disk of radius $r$ centered at the
origin.
\end{lem}

\begin{proof} As $\Bigm|{d\over dz}\sqrt{1+z}\Bigm|\le 1$ for all 
$|z|\le 3/4$,
by induction on $k$, $|w_1|+\cdots+|w_k|\le 3/4$ implies
$$\Biggm|\frac\partial{\partial w_i}\sqrt{\sqrt{\cdots\sqrt{\sqrt{1+w_1}+w_2}+\cdots+w_{k-1}}+w_k}\Biggm| \le 1$$
for $1\le i\le k$.
Thus,
$$\Biggm|\hbox{\vbox to 30pt{\vfill\vskip 16pt\hbox{$\sqrt{\sqrt{\cdots\sqrt{\sqrt{1+w}+2cz^{2^{n-1}}}
+\cdots+2cz^2}+2cz}-g_{n,c}(z)$}}}\Biggm|\le|w|,$$
whenever $|w|+|2cz^{2^{n-1}}|+\cdots+|2cz|\le 3/4$.  In particular,
$$|g_{n+1,c}-g_n(c)| \le |2cz^{2^n}|$$
provided
$$2|c|(|z|+|z|^2+|z|^4+\cdots+|z|^{2^n}).$$
It follows that the sequence
\begin{equation}
\label{gseq}
\{g_{k,c}(z)\}_{k=1,2,3,\ldots}
\end{equation}
converges whenever 
$$|z|< \inf(1,{3\over 14|c|}).$$  
By the
recursive definition (\ref{gnc}) of $g_{n,c}(z)$, the sequence (\ref{gseq}) converges for 
$z$ whenever $|z|\le r$ and it converges for $z^2$.  The convergence of (\ref{gseq}) in $\{z\colon |z|\le r\}$ follows by a bootstrapping argument.
\end{proof}

Let $R_c$ denote the minimum of $\lim_{n\to \infty} R_{n,c}$ and $1$.  We have the following immediate
corollary:

\begin{lem}
\label{gcconverge}
The series $g_c(z)$ converges for all $|z|<R_c$.
\end{lem}

In the next two results, we sketch a proof that there is an upper limit to the index of sparseness for any element of $\X$.

\begin{lem}
\label{nonzero}
Let $X$ be a compact set and $b_i\colon X\to\C$
a collection of continuous functions indexed by integers $i\ge 0$.
Let $f_x(z)=\sum_{k=0}^\infty b_k(x) z^k$.  We suppose that for each
$x\in X$ there exists $r_x>0$, depending continuously on $x$, such that $f_x(z)^2$ converges in a disk
of radius greater than $r_x$ and has exactly one zero, counting 
multiplicity, in the disk of radius $r_x$.  Then there exists
$N$ such that for all $k>N$ and for all $x\in X$, $b_k(x)\neq 0$.
\end{lem}

\begin{proof} By compactness we may assume without loss of generality that 
a single $r=r_x$ works for all $x\in X$.  Choose $s>r$ such that all
$f_x(z)^2$ have radius of convergence $>s$.  As $f_x$ is continuous in
$x$, the unique zero $z_x$ of $f_x(z)^2$ in the closed disk $D_r$ of radius $r$
varies continuously with $x$.  Therefore $f_x(z)^2\over z-z_x$ is 
continuous on $X\times D_r$ and nowhere vanishing on that set.  Therefore
its absolute value is always greater than some $\epsilon>0$.  
We make a branch cut from $z_x$ to $z_x\infty$ to make $f_x(z)$
single valued and then estimate $b_k(x)$ by computing the contour integral 
$\oint_{Q_x} {f_x(z)\over z^{k+1}}dz$, where $Q_x$ denotes a contour
consisting of an outward segment from $z_x$ to $s{z_x\over |z_x|}$,
a counterclockwise circle of radius $s$, and an inward segment from
$s{z_x\over |z_x|}$ to $z_x$.  For large values of $k$, only the two
segments matter, and their contributions are equal since $f_x(z)$
changes sign over the circle of radius $s$.  If 
$f_x(z)^2=c_1(z-z_x)+c_2(z-z_x)^2+\cdots$, the integral over one of
the segments of $Q_x$ is 
$$\Gamma(3/2)c_1^{1/2}z_x^{3/2}k^{-3/2}z_x^{-k}+O(k^{-5/2}z_x^{-k}).$$
Since $|c_1|>\epsilon$ and the implicit constant above is uniform in
$X$, $b_k(x)\neq 0$ for all $k\gg 0$ uniformly in $X$.
\end{proof}

\begin{thm}
For all open neighborhoods $U$ of the Mandelbrot set $\M$
there exists an integer $N$ such that for all $n>N$ and for all $c\not\in U$,
$M_n(-c/2)\neq 0$.
\end{thm}

\begin{proof} Making $U$ smaller if necessary, we may assume that it is bounded.
Let $U_1$ and $U_2$ be disjoint open sets in $\C\P^1$ such that $U_1$ contains
the complement of $U$ and $U_2$ contains $\M$.  
By construction the set $-{1\over 2} U_1$ satisfies the hypotheses of 
Lemma~\ref{single}
for all $n$ greater than some fixed $C$.  Let $K$ denote a compact
subset of $U_1$ containing the complement of $U$, and let $X$ denote the 
product of the one-point compactification $\Z^{\ge C}\cup\{\infty\}$ and 
$-{1\over 2}K$.  We define 
$$f_{n,c}(z)=\begin{cases} \sqrt{1-z}&\text{if } c=\infty,\\
                     g_c(z/c) &\text{if } n=\infty,\\
                     g_{n,c}(z/c) &\text{otherwise.}\\
                     \end{cases}
$$
By Lemma~\ref{convergence}, $f_x(z)$ is continuous in $x$ and is analytic in a
neighborhood of $0$ for each fixed $x$.  (Note that we have renormalized
the $g_{n,c}$ and $g_c$ to prevent the radius of convergence from going
to zero as $c\to \infty$.)  The conclusion of Lemma~\ref{single} implies that 
$f_x(z)$ satisfies the hypotheses of Lemma~\ref{nonzero}, and the theorem follows.
\end{proof}

By Proposition~\ref{integer}, if $f$ is 
sparse of index $m$, then $a_m$ is an algebraic 
integer.  On the other hand, $\X$ is rational over $\Q$, so all conjugates
of $a_m$ must also give rise to sparse solutions.  In particular, if $m\gg 0$,
$a_m$ and its conjugates all lie in any specified open set containing
$-{1\over 2}\M$.  Since $\M$ is a closed 
subset of the disk of radius $2$ meeting the boundary of the disk only at
the point $-2$, this open set can be taken to have capacity
less than $1$, and therefore to contain finitely many conjugacy classes of
algebraic integers \cite{Fekete}.  In fact, it is easy to see that it can be chosen
small enough that $0$ and $1$ are the only possible values for $a_m$.
The first is ruled out by definition, the second by Proposition~\ref{notone}.

However, to prove the main theorem, it is necessary to make the above
estimates effective.  We do this by choosing a particular open neighborhood
of $-{1\over 2}\M$, namely the disk of radius $7/8$ centered at $1/4$.
We begin by finding the orbits of algebraic integers belonging to this disk.

\begin{prop}
\label{outside-disk}
If $\alpha$ is an algebraic integer all of whose conjugates
satisfy $|z-1/4|\le 7/8$, then $\alpha$ is $0$ or $1$.
\end{prop}

\begin{proof} According to the maximum principle, for elements $\alpha_1,\ldots,\alpha_n$ 
of a closed disk of radius $r$,
the product $\prod_{i\neq j}|\alpha_i-\alpha_j|$ can 
achieve its maximum only if all $\alpha_i$ lie on the boundary of the disk.
By the concavity of $\log |1-e^{i\theta}|$, the product is achieved when the 
$\alpha_i$ form the vertices of an inscribed regular $n$-gon.  In this case,
the product is 
$$\LP \prod_{i=1}^{n-1} \bigm|r-r\zeta_n^i\bigm|\RP^n=r^{n^2-n}\Bigm|1+x+
x^2+\cdots+x^{n-1}\vert_{x=1}\Bigm|^n=n^nr^{n^2-n}.$$
For $r=7/8$, this expression is $<1$ for
$n\ge 26$.  For $6\le n\le 25$, we still have that $g(n)$ is less than the 
Minkowski bound ${n^{2n} \pi^n\over (n!)^2 4^n}$.  For $3\le n\le 5$,
$g(n)$ remains less than the smallest actual discriminant absolute value,
as tabulated in \cite{Bourbaki}.  Finally, for $n=2$, two conjugate algebraic
integers lie in the same disk of radius $7/8$ if and only if the 
integers are of the form $n+e^{\pm 2\pi i\over 3}$ for some $n\in \Z$.
In particular, no such pair lie in a disk centered at $1/4$.
\end{proof}

\begin{prop}
If $n\ge 2$ and $|d+1/2|>7/4$, then 
\begin{equation}
\label{final-ineq}
(|d+1/2|-1/4)^{2^{n-1}}<|I_n(d)|<(|d+1/2|+1/4)^{2^{n-1}}.
\end{equation}
\end{prop}

\begin{proof} As $I_2(d)=d^2+d=(d+1/2)^2-1/4$, setting $r=|d+1/2|>7/4$, we have
\begin{align*}
(r-\frac 14)^{2}+\frac 34(r-\frac 14)^{-1}&< r^2-\frac 14\le |I_2(d)|\\ 
&\le r^2+\frac 14 < (r+\frac 14)^{2}-\frac 34(r+\frac 14)^{-1}
\end{align*}
We prove by induction that for all $n\ge 2$, we have
\begin{equation}
\label{induction-hyp}
\begin{split}
(r-\frac 14)^{2^{n-1}}+\frac 34(r-\frac 14)^{1-2^{n-1}}&< r^2-\frac 14\le |I_n(d)| \\
&\le r^2+\frac 14 
< (r+\frac 14)^{2^{n-1}}-\frac 34(r+\frac 14)^{1-2^{n-1}}
\end{split}
\end{equation}
For the induction step, we apply 
$$|w|^2-r-\frac 12\le |w^2+d|\le |w|^2+r+\frac 12,$$
to $w = I_n(d)$ in \eqref{induction-hyp}, using
the inequalities
$$\frac 32(r\pm\frac 14) > r+\frac 14,\ \frac 34(r\pm \frac 14) > 1.$$
The inequalities \eqref{final-ineq} follow immediately.
\end{proof}

\begin{cor}
\label{Rc-bound}
If $n\ge 2$, $|c-1/4|>7/8$, then
$$\frac 1{2 |c-1/4|+1/4}<R_{n,c}<\frac 1{2 |c-1/4|-1/4} < \frac 23.$$
In particular, 
$${5\over 16|c|}\le R_c\le {3\over 4|c|}.$$
\end{cor}

\begin{proof} The proposition implies that $|I_n(-2c)|$ is monotonically increasing
for $n\ge 2$.
Thus $R_{n,c}=|I_n(-2c)|^{2^{1-n}}$.  The first claim follows immediately,
the second from the inequality
$${5\over 16} < {|c|\over |2|c-1/4|\pm 1/4|} < {3\over 4},$$
which holds for $|c-1/4| > 7/8$.
\end{proof}

\begin{cor}
\label{zc}
The sequence $z_{1,c},z_{2,c},\ldots$ converges to $z_c$, and $|z_c| = R_c$.
\end{cor}

\begin{proof}
By Lemma~\ref{single} and Corollary~\ref{Rc-bound}, $|z_{n,c}| = R_{n,c}$ for all $n\ge 1$.
As $g_{n,c}^2(z)$ converges for $|z| < R_{n-1,c}^{1/2}$, choosing
$$r\in ((2|c-1/4|-1/4)^{-1},(2|c-1/4|+1/4)^{-1/2}),$$
the sequence $g_{n,c}^2(z)$ converges to $g_c(z)^2$ uniformly on the disc $|z|\le r$, and $z_{n,c}$ is the unique zero of $g_{n,c}^2(z)$
in that disk and is moreover simple.  It follows that $z_c = \lim_{n\to \infty} z_c$.

\end{proof}

\begin{thm}
The only sparse elements in $\X$ are the linear solutions
(xi).
\end{thm}

\begin{proof} 
Suppose $f\in \X$ is sparse of index $p$.  Setting $c = a_p$, we have $c\neq 0$ by definition and $c\neq 1$ by Proposition~\ref{notone}.
Therefore, by Proposition~\ref{integer} and Proposition~\ref{outside-disk}, we have $|c - 1/2| > 7/4$..

We want to estimate the constant $N$ of Lemma~\ref{nonzero}.  
Our first task is to estimate the derivative of $h_c(z)=g_c(z)^2$
at its unique zero $z_c$ in the disk
$|z|<\lim_{n\to \infty} R_{n-1,c}^{1/2}$.  This is the same as the limit
of the derivative of $g_{n,c}(z)^2$ at its unique zero $z_{n,c}$ satisfying
$|z_{n,c}|\le R_{n-1,c}^{1/2}$.  
By induction on $k$, we have
$$g_{n-k,c}(z_{n,c}^{2^k}) = I_k(-2c)z_{n,c}^{2^{k-1}}$$
for $0\le k\le n$.  Differentiating \eqref{gnc}, we obtain
$$g'_{n-k,c}(z) =  \frac{c+z g_{n-k-1,c}'(z^2)}{g_{n-k,c}(z)}$$
for all $i\ge 1$, and substituting $z=z_{n,c}^{2^k}$, we get
$$g'_{n-k,c}(z_{n,c}^{2^k}) = \frac c{I_k(-2c)z_{n,c}^{2^{k-1}}} + \frac{z_{n,c}^{2^{k-1}}g'_{n-k-1,c}(z_{n,c}^{2^{k+1}})}{I_k(-2c)}.$$
Therefore, the value of the derivative of $g_{n,c}(z)^2 = 2cz+ g_{n-1,c}(z^2)$ at $z_{n,c}$ is
\begin{align*}
2c+2z_{n,c}g'_{n-1,c}(z_{n,c}^2)
&=2c+\frac{2c}{I_1(-2c)}
+\frac{2z_{n,c}^2g'_{n-2,c}(z_{n,c}^{4})}{I_1(-2c)} \\
&=2c+\frac{2c}{I_1(-2c)}+\frac{2c}{I_1(-2c)I_2(-2c)}+\frac{2z_{n,c}^4g'_{n-3,c}(z_{n,c}^{8})}{I_2(-2c)} \\
&=\cdots.
\end{align*}
Expanding completely (and using the fact that $g'_{0,c}$ is identically zero),
we obtain
$$2c\LP 1+{1\over I_1(-2c)}+{1\over I_1(-2c)I_2(-2c)}+\cdots+{1\over I_1(-2c)
  I_2(-2c)\cdots I_{n-1}(-2c)}\RP.$$
As $I_1(-2c)=-2c$ lies on a circle of radius $7/4$ centered at $-1/2$, its
inverse lies on the circle with diameter the real interval $[-4/9,4/5]$.
It follows that $|1+I_1(-2c)^{-1}|\ge 5/9$.  On the other hand,
$|I_1(-2c)|\ge 5/4$, and by  \eqref{final-ineq},
$|I_2(-2c)|\ge 9/4$, and $|I_n(-2c)|\ge 5$ for $n\ge 3$, 
so 
$$\Bigm|1+{1\over I_1(-2c)}+\cdots+{1\over I_1(-2c)
  \cdots I_{n-1}(-2c)}\Bigm|
  \ge 5/9-{1+5^{-1}+5^{-2}+\cdots\over |I_1(-2c) I_2(-2c)|}\ge {1\over 9}.$$
Thus, 
\begin{equation}
\label{hcprime}
|h'_c(z_c)|\ge {|c|\over 9}>{1\over 30R_c}.
\end{equation}

Next, we need to estimate the second derivative of $h_c(z)$ near $z=z_c$.
By Cauchy's integral formula for derivatives,
$$|f''(z)|\le 2\frac{\sup_{\theta}|f(z+re^{i\theta})|} {r^2}.$$
By Lemma~\ref{gcconverge}, $h_c(z)$ converges for $|z|<\sqrt{R_c}$ and therefore, by Corollary~\ref{Rc-bound},
for $|z|<1.2R_c$.  For $|z|<1.1R_c$, we may take $r=R_c/10$ and still have
$|c(z+re^{i\theta})|<1$ by Corollary~\ref{Rc-bound}.  As $|c|>1$, the inequality $|\sqrt{1+z}|\le 1+|z|/2$ implies
\begin{equation}
\label{g-bound}
|g_{n,c}(z+re^{i\theta})|\le 1+|c(z+re^{i\theta})|+\frac{|c(z+re^{i\theta})^2|}2+\frac{|c(z+re^{i\theta})^4|}4+\cdots\le 3.
\end{equation}
Thus, $|z| < 1.1 R_c$ implies
$$|h_c''(z)|\le {1800\over R_c^2}.$$
By \eqref{hcprime}, $|z-z_c|\le R_c/120$ implies
$$\biggm|{h'_c(z)\over h'_c(z_c)}-1\biggm|\le {1\over 2},$$
so integrating $h'_c(z)$ along the directed line segment from $z_c$ to $z$, we obtain
$$\biggm|\frac{h_c(z)}{h'_c(z_c)(z-z_c)}-1\biggm|
= \biggm|\frac{\int_{z_c}^z h'_c(w)dw}{h'_c(z_c)(z-z_c)}-1\biggm|
\le {1\over 2}.$$
As $|(x+iy)^2-1| \le 1/2$ implies
$$(x^2+y^2)^2 +2y^2 + 1 - 2x^2 = (x^2-y^2-1)^2 + (2xy)^2 \le \frac 14,$$
it follows that 
$$\Re \sqrt\frac{h_c(z)}{h'_c(z_c) (z-z_c)} > \frac 12$$
in the ball $|z-z_c|\le R_c/120$.

We integrate $\sqrt{h_c(z)}z^{-k-1}$ over the contour consisting of a
straight line from $z_c$ to ${121\over 120}z_c$, a counterclockwise circle
$C$ of radius $\frac{121R_c}{120}$, and a straight line returning to $z_c$.
As $\sqrt{h_c(z)}$ changes sign over the contour, the integral is twice
the original segment plus the circle.  We will show that the integral is 
non-zero by showing that 
\begin{equation}
\label{Re}
\Re\LP h'_c(z_c)^{-1/2}z_c^{k+1/2}\int_{z_c}^{121z_c\over 120}{\sqrt{h_c(z)}\over z^{k+1}}dz\RP>
  \biggm| h'_c(z_c)^{-1/2}z_c^{k+1/2}\int_C {\sqrt{h_c(z)}\over z^{k+1}}dz\biggm|.
\end{equation}
The left hand side of \eqref{Re} is the integral of 
\begin{equation}
\label{integrand}\Bigl(\frac{z_c}{z}\Bigr)^k \sqrt{\frac{z_c(z-z_c)}{z^2}}\Re \sqrt{\frac{h_c(z)}{h'_c(z_c)(z-z_c)}}
> \frac 12 \Bigl(\frac{z_c}{z}\Bigr)^k \sqrt{\frac{z_c(z-z_c)}{z^2}}.
\end{equation}
It is therefore 
greater than the integral of the right hand side of \eqref{integrand} from $z_c(1+1/480)$ to $z_c(1+1/240)$,
and so is at least $\frac{\sqrt{480}(1+1/240)^{-k}R_c}{2\cdot 480 \cdot 481}$.  The right hand side of \eqref{Re}
is no larger than 
$$2\pi \frac{121R_c}{120}\frac 1{\sqrt{|h'_c(z_c)R_c|}}\frac{\sup_{z\in C}\sqrt{|h_c(z)|}}{(1+1/120)^{k+1}}
\le 2\pi \frac{121R_c}{120}\frac 1{\sqrt30}\frac 3{(1+1/120)^{k+1}},$$
by \eqref{hcprime} and \eqref{g-bound}.
For $k \ge 2773$, we have
$$\Big(\frac{242}{241}\Bigr)^k > 48\cdot 481\pi,$$
which implies \eqref{Re}.

If $p\ge 2^{13}-1$, then by Lemma~\ref{AP}, there exists $k$ satisfying $p-1 \ge k \ge (p-1)/2 \ge 2773$ such that $k(p-1)+1\not\in \P$ and
the $z^k$ coefficient of $g_c(z)$, i.e., $M_k(c) = M_k(a_p)$ is non-zero.  This contradicts Proposition~\ref{multizero}, and we are done.

This leaves two cases: $p=31$ and $p=127$.  For the former, $3\cdot 30+1\not\in \P$ and for the latter, $2\cdot 126+1\not\in\P$.
Now, either $M_2(c) = 0$ or $M_3(c) = 0$ implies $c\in \{0,1\}$, which is impossible.  This again contradicts Proposition~\ref{multizero}, which proves the theorem.

\end{proof}

\section{Some Variants}

In this section, we consider some variants of the problem of classifying 
normalized multiplicative power series whose squares are multiplicative.

We begin by proving Theorem~\ref{four}, or more precisely:

\begin{prop}
The set of normalized multiplicative power series $f(q)$ such that $f(q)^2$ and $f(q)^4$ are
both multiplicative is as follows:
\begin{equation}
\label{solns}
\{\vartheta_{\Z}(q),\vartheta_{\Z[i]}(q),\vartheta_{\Z[\zeta_3]}(q), -\vartheta_{\Z}(-q),-\vartheta_{\Z[i]}(-q),-\vartheta_{\Z[\zeta_3]}(-q)\}.
\end{equation}
\end{prop}

\begin{proof}
First, we claim that each series $f(q)$ in (\ref{solns}) is a solution.  It suffices to prove that $f(q)$ and some multiple of $f(q)^2$ lie in $\X$, and by Lemma~\ref{SignChange}, it suffices to prove
this for $\vartheta_{\Z}(q)$, $\vartheta_{\Z[i]}(q)$, and $\vartheta_{\Z[\zeta_3]}(q)$.
By Proposition~\ref{KnownSolutions}, these are of type (vii), (v), and (vi) respectively.  As $\vartheta_{\Z}(q)^2 = \vt_{\Z[i]}(q)$
and $2\vt_{\Z[i]}^2(q) = \vt_{H}(q)+2\vt_{H}(q^2)$, the squares of the theta series, suitably normalized, are elements of $\X$ of type (v), (iii), and (ii) respectively.

Let the polynomials $P_n$ be defined as in (\ref{poly-P}).  We define polynomials $Q_n$ which play the role for $f^4$ which the $P_n$ play for $f^2$; namely, if $2a_0^3 f(q)^4 = \sum_n d_nq^n$, and $D_n$ denotes the polynomial expression
in $a_0,a_2,\ldots,$ for the coefficient $d_n$, we set
$$Q_{p_1^{e_1}\cdots p_k^{e_k}} = D_{p_1^{e_1}\cdots p_k^{e_k}} - D_{p_1^{e_1}}\cdots D_{p_k^{e_k}}.$$
We consider the system of $14$ polynomial equations in the $13$ variables $a_0,a_2,\ldots,a_{19}$ given by $P_n$ and $Q_n$ for $n\in [6,20]\cap\Z\setminus \P$.

A Maple computation shows that there are exactly six solutions,
corresponding to the initial coefficients of the six modular forms listed above.  Since performing this calculation reasonably efficiently is not straightforward, we describe our steps in more detail.
We begin by solving for the variables $a_4$, $a_5$, $a_8$, $a_9$, $a_{11}$, $a_{13}$, $a_{16}$, $a_{17}$, $a_{19}$ 
using the polynomial equations $Q_6$, $P_6$, $Q_{10}$, $P_{10}$, $P_{12}$, $P_{14}$, $Q_{18}$, $P_{18}$, and $P_{20}$ respectively and substituting the resulting expressions into the equations $Q_{12}$, $Q_{14}$, $P_{15}$, $Q_{15}$, $Q_{20}$.
The resulting polynomials in $a_0$, $a_2$, $a_3$, and $a_7$ have degrees $11$, $11$, $13$, $13$, and $17$ respectively.
We reduce to equations in $a_0$ and $a_2$ by using $Q_{12}$ to eliminate $a_7$ and $Q_{14}$ to eliminate $a_3$ from $P_{15}$, $Q_{15}$, $Q_{20}$.  
These three equations have a degree $24$ common factor, $A(a_0,a_2)^2$, but pulling out this factor and using the first of the three remaining factors to eliminate $a_0$ from the second and third,
we can take g.c.d.
to solve for $a_2$.  The possible solutions, $0$, $\pm 1$, $\pm \frac 12$ can then be substituted back into the original equations $Q_{12}$, $Q_{14}$, $P_{15}$, $Q_{15}$, $Q_{20}$, at which point Maple is capable of solving directly for
all triples $(a_0,a_3,a_7)$.  To deal with solutions of $A(a_0,a_2)=0$, we eliminate $a_7$ and $a_3$ from $Q_{15}$ and $Q_{20}$ using $Q_{14}$ and $P_{15}$ respectively.  The resulting polynomials in $a_0$ and $a_2$ again have a common factor, $B(a_0,a_2)^4$,
of degree $92$.  Removing this factor from $Q_{15}$ and $Q_{20}$ and eliminating $a_0$ using $A$, we see again that $a_2\in \{0,\pm 1,\pm \frac 12\}$.  Thus, we need only consider the case $A(a_0,a_2)=B(a_0,a_2)=0$.  Eliminating $a_7$ and $a_3$ from $Q_{20}$ using $P_{15}$ and $Q_{15}$ respectively,
we obtain an equation in $a_0$ and $a_2$, and eliminating $a_2$ from this equation and $B$ using $A$, we get $a_0=0$, which is impossible.

By Proposition~\ref{ExceptSparse}, there is at most one solution $f(q)$ with each of these initial coefficient
sequences.

\end{proof}

Note that Theorem~\ref{eight}, or more precisely, the following statement, is an immediate corollary:

\begin{cor}
The set of normalized multiplicative power series $f(q)$ such that $f(q)^2$, $f(q)^4$, and $f(q)^8$ are
all multiplicative consists of 
$$\{\vt_{\Z}(q),-\vt_{\Z}(-q)\}.$$

\end{cor}

Next we consider the following question:  What can be 
said about $f(q)$ if $f$ and $f^2$ both belong to the vector space $V$ 
of \emph{finite linear combinations} of multiplicative power series, or more 
generally, if all powers of $f$ belong to $V$?  The following proposition 
proves that this question is not vacuous.  

\begin{prop}
The vector space $V$ is a proper subspace of the 
complex power series in $q$.
\end{prop}

\begin{proof} 
We prove the following stronger claim:  There exists a function $F(x)$ such that if  $|a_{n+1}| \ge F(|a_n|)$ for all $n\ge 0$, then
$f(q)=\sum_{n=0}^\infty a_n q^n$ does not belong to $V$.  

Suppose 
$$f(q)=a_0+\sum_{i=1}^n c_i f_i(q),$$
where the $f_i$ are normalized multiplicative :
$$f_i(q)=a_{i,0}+q+a_{i,2}q^2+a_{i,3}q^3+a_{i,4}q^4+a_{i,5}q^5+a_{i,2}a_{i,3}q^6
+\cdots.$$
Let $C_k(x_i,y_{i,j})$ denote the polynomial representing the $q^k$ 
coefficient of $f$ in terms of $x_i=c_i$ and $y_{i,j}=a_{i,j}$ ($j\in\P\cup\{0\}$).  Thus 
$C_k$ is a sum of distinct products of subsets of the variables 
$$\{x_i\mid 1\le i\le n\}\cup\{y_{i,j}\mid 1\le i\le n,\ j\le k,\ 
j\in\P\cup\{0\}\}.$$
By the prime number theorem, the number of variables in the set grows 
like $nk/\log k$.  Therefore, for $N\gg 0$, the polynomials 
$C_{N+1},\,C_{N+2},\,\ldots,\,C_{2N}$ involve among them fewer than $N$ 
variables.  The proposition now follows from the following two lemmas:  
\end{proof}

\begin{lem}
There exist functions $G,H\colon \N\to \R$ such that if
$$Q_i(x_1,\ldots,x_m)=\sum_{I\in\{0,1\}^m} a_{i,I}x^I,\ 
i=1,\ldots,m+1,$$
with $a_{i,I}\in\{0,1\}$, then there exists a polynomial 
$R(y_1,\ldots,y_{m+1})$ of degree $\le G(m)$ and integer coefficients 
of absolute value $\le H(m)$ such that 
$$R(Q_1(x),\ldots Q_{m+1}(x))\equiv 0.$$
\end{lem}

\begin{proof}
For any positive integer $N$,
there are $\binom{N+m+1}{m+1}$ monomials in $Q_1,\ldots,Q_{m+1}$ of degree $\le N$;
all are of degree $\le mN$ as polynomials in $x_1,\ldots,x_m$ and have all coefficients $\le (2^m)^N$.  The total number of monomials
of degree $\le mN$ in the $x_i$ is $\binom{mN+m}{m}$.  If $N = G(m)$ is sufficiently large, the former number is larger, so there must be some 
linear relation between the monomials, and the coefficients can be bounded by $H(m)$ depending only on $N$ and $m$, and therefore only on $m$.

 \end{proof}

\begin{lem}
Given functions $G,H\colon \N\to \N$ there exists a function $F\colon \R\to \R$ such that if $m\ge 2$, $z_1,\ldots,z_{2m}\in \C$
satisfy $|z_{i+1}|\ge F(|z_i|)$ for $1\le i\le 2m-1$, and $R\in \Z[x_1,\ldots,x_m]$ is a non-zero polynomial with degree $\le G(m)$ and coefficients with absolute value $\le H(m)$, then
$$R(z_{m+1},\ldots,z_{2m})\neq 0.$$
\end{lem}

\begin{proof}
Let $G^*(x)$ (resp. $H^*(x)$) denote the maximum value of $G(n)$ (resp. $H(n)$) for $n\le x$.  Replacing $G$ and $H$ by
$G^*$ and $H^*$ respectively, we may regard both as non-decreasing functions defined on $[0,\infty)$.  Let
$$F(x) =(1+H(|x|)) e^{(1+G(|x|))x}+3.$$
Then we have $F(x) \ge e^x+3 >  \max(3,x+3)$.  By induction on $r$, we have $|z_{r+1}|\ge 3r\ge r+2$ for all $r\ge 1$, so
$$|z_{r+2}| > F(|z_{r+1}|) > (1+H(r+2))e^{(1+G(r+2))|z_{r+1}|}.$$
For all $x,a\ge 0$, we have $e^{ax} = (e^x)^a \ge (x^e)^a = x^{ea},$ so
$$|z_{r+2}| > H(r+2) |z_{r+1}|^{e(1+G(r+2))} = H(r+2) |z_{r+1}|^{G(r+2)}(r+2)^{1+G(r+2)}.$$
In particular, for $1\le j\le m$,
$$|z_{m+j}| > (1+G(m))^n H(m) \prod_{i=1}^{j-1} |z_{m+i}|^{G(m}.$$
Thus, given $m$-tuples of non-negative integers $\le G(m)$ such that 
$$(k_1,\ldots,k_m)>(k'_1,\ldots,k'_m)$$
in lexicographic order, we have
$$|z_{2m}|^{k_1}\cdots |z_{m+2}|^{k_{m-1}}|z_{m+1}|^{k_m} >  (1+G(m))^n H(m) |z_{2m}|^{k'_1}\cdots |z_{m+2}|^{k'_{m-1}}|z_{m+1}|^{k'_m}$$
which in turn implies that any non-trivial integer linear combinations of monomials $z_{2m}^{k_1}\cdots z_{m+1}^{k_m}$ with $k_i\le G(m)$ and coefficient
absolute values $\le H(m)$ is non-zero.
\end{proof}

If $a_n$ and $b_n$ are multiplicative sequences, then
the sequences $n\mapsto a_n b_n$ and $n\mapsto \sum_{ij=n} a_i b_j$ are multiplicative.
The polynomial
$$S_n(q)=\sum_{d\mid n} q^d$$
has multiplicative coefficients, and every polynomial, in particular, 
every monomial in $q$ is a finite linear combination of the polynomials
$S_i$.  It follows that $f(q^n)\in V$ whenever $f(q)\in V$.

If $M_*(N)$ denotes the graded ring of modular forms of integral
weight for $\Gamma_1(N)$, then it is clear by reduction to the case of
newforms that $\bigcup_N 
M_*(N)\subset V$.  As the union
of the $M_*(N)$ is a ring, the same is true for all powers of $f$.  Certain
power series, such as $24E_2(q)$, though not modular forms themselves, are 
congruent to elements of $M_*(N)$ modulo every prime \cite{Congruences}.  Naturally,
any integer power of such a series has the same property.

\begin{que}Is $E_2(q)^2\in V$?
\end{que}

In a different direction, we have the following:

\begin{prop}
If $f(q)$ is the $q$-expansion of a modular form of weight
$1/2$, then $f$ and $f^2$ are both in $V$.
\end{prop}

\begin{proof} Obviously $f^2\in M_*(N)$ for some $N$, so $f^2\in V$. 
As for $f$, by \cite{SerreStark}, it is a finite linear combination
of series of the form 
$$\sum_{n=-\infty}^{\infty}\psi(n)q^{kn^2},$$
where $k$ is a positive integer and $\psi$ is periodic.  Equivalently, $f$
is a linear combination of series 
$$f_{k,m,a}=\sum_{n\in a+m\Z} q^{kn^2}=\begin{cases} 
         \frac 12{\displaystyle\mathop{\sum}_{\con{n}{\pm a}{m}}q^{kn^2}} &\text{if } m\neq 2,\\
         {\displaystyle\mathop{\sum}_{\con{n}{a}{2}}q^{kn^2}} &\text{if } $m=2$,
\end{cases}$$
where $(a,m)=1$.  

It therefore suffices to prove that $f_{1,m,a}\in V$ whenever $a$ 
and $m$ are relatively prime, but this is clear since $f_{1,m,a}$ is a 
linear combination of the multiplicative power series 
$\sum_{n\in\Z}\chi(n)q^{n^2}$, as $\chi$ ranges over the even characters of
$(\Z/m\Z)^*$.
\end{proof}

\begin{que}Are forms of half-integral weight $k\ge 3/2$ finite
linear combinations of multiplicative power series?
\end{que}

\vfill\eject
\section{Appendix, by Anne Larsen}

A computer was used to find all multiplicative series whose squares are also multiplicative (when multiplied by a suitable scalar), mod small primes. The series found, excluding mod $p$ versions of the general types listed before and ``sparse'' solutions, are listed in the tables below. However, for each solution, the same series with even coefficients multiplied by $-1$ will also be a solution; only one solution in each pair is exhibited in the tables.

Note that there is no table of mod $2$ solutions because $$(1+a_1q+a_2q^2+\ldots)^2 = 1+(a_1)^2q^2+(a_2)^2q^4+\ldots \pmod 2,$$ which has no $q$ term and is therefore not strictly a multiplicative series. 

For all but one (mod $3$) solution, the comments column gives a possible match for the series as some modular form. Usually, the series is identified as a modified Eisenstein or $\theta$-series. (The modification consists of taking some finite linear combination of $f(q^k)$, so, for example, the mod $19$ solution listed as $E_6$ is actually $E_6(q) - E_6(q^2) + 7E_6(q^4)$.) However, there are also some cusp forms (plus a scalar term, which is interpreted as $E_{p-1}$), which are identified by their label on the online LMFDB database of holomorphic cusp forms. Proposition 3.7 provides a proof that the mod $13$ series identified as $1.12.1.a$ in the tables is indeed multiplicative; presumably, proofs for the other cusp forms should be similar.

There were no exceptional solutions mod 23, 29, or 31.

\begin{table}
\caption{Exceptional solutions (mod 3)}
\center
\begin{tabular}{|c|c|c|c|c|c|c|c|c|c|c|c|} 
\hline 
$1/2a_0$ & $a_2$ & $a_3$ & $a_4$ & $a_5$ & $a_7$ & $a_8$ & $a_9$ & $a_{11}$ & $a_{13}$ & $a_{16}$ & Comments\\ 
\hline 
$1$ & $0$ & $0$ & $0$ & $0$ & $2$ & $0$ & $0$ & $0$ & $2$ & $0$& $E_2$ \\ 
\hline 
$1$ & $0$ & $1$ & $0$ & $0$ & $2$ & $0$ & $1$ & $0$ & $2$ & $0$& $E_2$ \\ 
\hline 
$1$ & $0$ & $1$ & $1$ & $0$ & $1$ & $0$ & $1$ & $0$ & $2$ & $1$& $E_2$ \\ 
\hline 
$1$ & $0$ & $2$ & $2$ & $2$ & $0$ & $0$ & $0$ & $1$ & $0$ & $2$& $\vartheta_{\Z[\frac{1+\sqrt{-11}}2]}$ \\ 
\hline 
$1$ & $1$ & $1$ & $0$ & $0$ & $1$ & $1$ & $1$ & $0$ & $2$ & $0$& $E_2$\\ 
\hline 
$1$ & $1$ & $1$ & $0$ & $0$ & $2$ & $0$ & $1$ & $0$ & $2$ & $0$& $E_2$ \\ 
\hline 
$1$ & $1$ & $1$ & $1$ & $2$ & $2$ & $1$ & $1$ & $0$ & $2$ & $1$& $E_{2}$ \\ 
\hline 
$1$ & $1$ & $1$ & $2$ & $1$ & $2$ & $1$ & $1$ & $0$ & $2$ & $2$& $E_{2}$ \\ 
\hline 
$1$ & $1$ & $2$ & $1$ & $1$ & $0$ & $0$ & $1$ & $2$ & $1$ & $1$& $15.2.1.a$\\ 
\hline 
$1$ & $1$ & $2$ & $2$ & $0$ & $0$ & $2$ & $0$ & $2$ & $0$ & $2$& $\vartheta_{\Z[\sqrt2]}$\\ 
\hline 
$1$ & $2$ & $1$ & $0$ & $2$ & $2$ & $2$ & $1$ & $0$ & $2$ & $0$& $E_{2}$ \\ 
\hline 
$1$ & $2$ & $1$ & $1$ & $0$ & $0$ & $0$ & $1$ & $2$ & $2$ & $1$& $75.2.1.a$ or $75.2.1.b$\\ 
\hline 
$1$ & $2$ & $1$ & $2$ & $1$ & $2$ & $0$ & $1$ & $1$ & $1$ & $2$& $21.2.1.a$ \\ 
\hline 
$1$ & $2$ & $2$ & $0$ & $1$ & $1$ & $2$ & $1$ & $1$ & $1$ & $2$& \\ 
\hline 
$1$ & $2$ & $2$ & $2$ & $0$ & $1$ & $2$ & $1$ & $0$ & $1$ & $2$& $50.2.1.b$\\ 
\hline 
\end{tabular} 
\end{table} 

\begin{table}
\caption{Exceptional solutions (mod 5)}
\center
\begin{tabular}{|c|c|c|c|c|c|c|c|c|c|c|c|} 
\hline 
$1/2a_0$ & $a_2$ & $a_3$ & $a_4$ & $a_5$ & $a_7$ & $a_8$ & $a_9$ & $a_{11}$ & $a_{13}$ & $a_{16}$ & Comments\\ 
\hline 
$1$ & $1$ & $2$ & $2$ & $1$ & $4$ & $0$ & $4$ & $2$ & $3$ & $4$& $E_{4}$ \\ 
\hline 
$1$ & $1$ & $2$ & $3$ & $0$ & $1$ & $0$ & $2$ & $2$ & $2$ & $1$& $5.4.1.a$ \\ 
\hline 
$1$ & $2$ & $3$ & $3$ & $1$ & $3$ & $0$ & $3$ & $2$ & $4$ & $4$& $E_{2}$ \\ 
\hline 
$1$ & $2$ & $4$ & $4$ & $1$ & $3$ & $4$ & $3$ & $2$ & $4$ & $4$& $E_{2}$ \\ 
\hline 
$1$ & $3$ & $4$ & $2$ & $1$ & $3$ & $0$ & $3$ & $2$ & $4$ & $1$& $E_{2}$ \\ 
\hline 
$2$ & $3$ & $3$ & $4$ & $1$ & $4$ & $2$ & $2$ & $2$ & $3$ & $1$& $E_{4}$ \\ 
\hline 
\end{tabular} 
\end{table} 

\begin{table}
\caption{Exceptional solutions (mod 7)}
\center
\begin{tabular}{|c|c|c|c|c|c|c|c|c|c|c|c|} 
\hline 
$1/2a_0$ & $a_2$ & $a_3$ & $a_4$ & $a_5$ & $a_7$ & $a_8$ & $a_9$ & $a_{11}$ & $a_{13}$ & $a_{16}$ & Comments\\ 
\hline 
$1$ & $6$ & $2$ & $3$ & $6$ & $2$ & $0$ & $4$ & $3$ & $1$ & $2$& $3.6.1.a$ \\ 
\hline 
$2$ & $5$ & $2$ & $5$ & $2$ & $0$ & $5$ & $1$ & $0$ & $2$ & $5$& $\vartheta_{\Z[i]}$ \\ 
\hline 
$3$ & $2$ & $4$ & $1$ & $1$ & $1$ & $6$ & $6$ & $5$ & $0$ & $2$& $E_{2}$ \\ 
\hline 
$3$ & $2$ & $5$ & $2$ & $2$ & $0$ & $2$ & $1$ & $0$ & $2$ & $2$& $\vartheta_{\Z[i]}$ \\ 
\hline 
$3$ & $3$ & $0$ & $2$ & $0$ & $0$ & $3$ & $2$ & $1$ & $0$ & $4$& $7.6.1.a$ \\ 
\hline 
$3$ & $3$ & $6$ & $5$ & $4$ & $1$ & $6$ & $3$ & $3$ & $0$ & $3$& $E_{6}$ \\ 
\hline 
\end{tabular} 
\end{table} 

\begin{table}
\caption{Exceptional solutions (mod 11)}
\center
\begin{tabular}{|c|c|c|c|c|c|c|c|c|c|c|c|} 
\hline 
$4$ & $3$ & $1$ & $10$ & $0$ & $2$ & $3$ & $1$ & $0$ & $2$ & $10$& $\vartheta_{\Z[\zeta_3]}$ \\ 
\hline 
$5$ & $6$ & $9$ & $8$ & $1$ & $5$ & $7$ & $0$ & $7$ & $4$ & $2$& $2.10.1.a$\\ 
\hline 
\end{tabular} 
\end{table} 

\begin{table}
\caption{Exceptional solutions (mod 13)}
\center
\begin{tabular}{|c|c|c|c|c|c|c|c|c|c|c|c|} 
\hline 
$1/2a_0$ & $a_2$ & $a_3$ & $a_4$ & $a_5$ & $a_7$ & $a_8$ & $a_9$ & $a_{11}$ & $a_{13}$ & $a_{16}$ & Comments\\ 
\hline 
$2$ & $2$ & $5$ & $10$ & $7$ & $0$ & $6$ & $3$ & $0$ & $8$ & $7$& $1.12.1.a$ \\ 
\hline 
$2$ & $5$ & $11$ & $10$ & $9$ & $6$ & $11$ & $8$ & $6$ & $1$ & $6$& $E_{4}$ \\ 
\hline 
$4$ & $6$ & $10$ & $9$ & $6$ & $12$ & $1$ & $0$ & $8$ & $1$ & $5$& $E_{6}$ \\ 
\hline 
\end{tabular} 
\end{table} 

\begin{table}
\caption{Exceptional solutions (mod 17)}
\center
\begin{tabular}{|c|c|c|c|c|c|c|c|c|c|c|c|} 
\hline 
$1/2a_0$ & $a_2$ & $a_3$ & $a_4$ & $a_5$ & $a_7$ & $a_8$ & $a_9$ & $a_{11}$ & $a_{13}$ & $a_{16}$ & Comments\\ 
\hline 
$3$ & $12$ & $1$ & $16$ & $5$ & $14$ & $16$ & $12$ & $5$ & $7$ & $14$& $1.16.1.a$\\ 
\hline 
$4$ & $7$ & $12$ & $11$ & $11$ & $13$ & $13$ & $14$ & $4$ & $5$ & $14$& $E_{8}$\\ 
\hline 
\end{tabular} 
\end{table} 

\begin{table}
\caption{Exceptional solutions (mod 19)}
\center
\begin{tabular}{|c|c|c|c|c|c|c|c|c|c|c|c|} 
\hline 
$1/2a_0$ & $a_2$ & $a_3$ & $a_4$ & $a_5$ & $a_7$ & $a_8$ & $a_9$ & $a_{11}$ & $a_{13}$ & $a_{16}$ & Comments\\ 
\hline 
$1$ & $5$ & $15$ & $5$ & $2$ & $0$ & $5$ & $1$ & $0$ & $2$ & $5$& $\vartheta_{\Z[i]}$ \\ 
\hline 
$5$ & $6$ & $15$ & $6$ & $2$ & $0$ & $6$ & $1$ & $0$ & $2$ & $6$& $\vartheta_{\Z[i]}$ \\ 
\hline 
$5$ & $13$ & $16$ & $5$ & $10$ & $12$ & $15$ & $13$ & $8$ & $15$ & $12$& $E_{6}$ \\ 
\hline 
$9$ & $7$ & $15$ & $7$ & $2$ & $0$ & $7$ & $1$ & $0$ & $2$ & $7$& $\vartheta_{\Z[i]}$ \\ 
\hline 
\end{tabular} 
\end{table} 

\end{document}